\documentclass[12pt]{article}
\usepackage[margin = 20mm]{geometry}
\usepackage{amsmath}
\usepackage{amsfonts}
\usepackage{mathrsfs}
\usepackage{amssymb}
\usepackage{amsthm}
\usepackage{cases}
\usepackage{enumerate}
\usepackage{tocloft}
\usepackage{setspace}
\usepackage{abstract}
\usepackage{nicematrix}
\usepackage{graphicx}
\usepackage{xcolor, colortbl}
\usepackage[ruled]{algorithm2e}
\usepackage{tikz}
\usepackage{multirow}
\usepackage{caption}
\usetikzlibrary{shapes.geometric}
\usetikzlibrary{graphs}
\usetikzlibrary{graphs.standard}
\usepackage[backend = biber, doi = false, url = false, isbn = false, maxbibnames = 10]{biblatex}
\renewbibmacro{in:}{}
\renewbibmacro*{issue+date}{%
	\ifboolexpr{not test {\iffieldundef{year}} or not test {\iffieldundef{issue}}}
	{\printtext[parens]{%
			\iffieldundef{issue}
			{\usebibmacro{date}}
			{\printfield{issue}%
				\setunit*{\addspace}%
				\usebibmacro{date}}}}
	{}%
	\newunit}
\newtheorem{thm}{Theorem}[section]
\newtheorem{cor}[thm]{Corollary}
\newtheorem{lem}[thm]{Lemma}

\theoremstyle{definition}
\newtheorem{defin}[thm]{Definition}

\numberwithin{equation}{section}
\def\F{\mathbb{F}}
\def\L{\mathcal{L}}
\def\E{\mathcal{E}}
\def\G{\Gamma}
\def\R{\mathcal{R}}
\def\N{\mathcal{N}}

\def\B{\mathcal{B}}
\def\pbd{\text{PBD}}
\def\eref#1{$(\ref{#1})$}
\def\sref#1{\S$\ref{#1}$}
\def\lref#1{Lemma~$\ref{#1}$}
\def\tref#1{Theorem~$\ref{#1}$}
\def\cyref#1{Corollary~$\ref{#1}$}

\def\dref#1{Definition~$\ref{#1}$}

\renewcommand{\geq}{\geqslant}
\renewcommand{\leq}{\leqslant}
\renewcommand{\emptyset}{\varnothing}
\makeatletter
\g@addto@macro\bfseries{\boldmath}
\makeatother
\title{Cycles of quadratic Latin squares and anti-perfect $1$-factorisations}
\author{Jack Allsop\\
	\small School of Mathematics\\[-0.5ex]
	\small Monash University\\[-0.5ex]
	\small Vic 3800, Australia\\
	\small\tt jack.allsop@monash.edu}
\date{}

\begin{document}
	\maketitle
	
	\begin{abstract}
		A Latin square of order $n$ is an $n \times n$ matrix of $n$ symbols, such that each symbol occurs exactly once in each row and column. For an odd prime power $q$ let $\F_q$ denote the finite field of order $q$. A quadratic Latin square is a Latin square $\mathcal{L}[a, b]$ defined by
		\[
		(\mathcal{L}[a, b])_{i, j} = \begin{cases}
			i + a(j-i) & \text{if } j-i \text{ is a quadratic residue in } \F_q, \\
			i + b(j-i) & \text{otherwise},
		\end{cases}
		\]
		for some $\{a, b\} \subseteq \F_q$ such that $ab$ and $(a-1)(b-1)$ are quadratic residues in $\F_q$. Quadratic Latin squares have previously been used to construct perfect $1$-factorisations, mutually orthogonal Latin squares and atomic Latin squares. We first characterise quadratic Latin squares which are devoid of $2 \times 2$ Latin subsquares. Let $G$ be a graph and $\mathcal{F}$ a $1$-factorisation of $G$. If the union of every pair of $1$-factors in $\mathcal{F}$ induces a Hamiltonian cycle in $G$ then $\mathcal{F}$ is called perfect, and if there is no pair of $1$-factors in $\mathcal{F}$ which induce a Hamiltonian cycle in $G$ then $\mathcal{F}$ is called anti-perfect.
		We use quadratic Latin squares to construct new examples of anti-perfect $1$-factorisations of complete graphs and complete bipartite graphs. We also demonstrate that for each odd prime $p$, there are only finitely many orders $q$, which are powers of $p$, such that quadratic Latin squares of order $q$ could be used to construct perfect $1$-factorisations of complete graphs or complete bipartite graphs.
		
		\medskip
		
	\end{abstract}
	
	\section{Introduction}\label{s:intro}
	
	A \emph{Latin rectangle} is an $n \times m$ matrix, with $n \leq m$, on $m$ symbols such that each symbol occurs at most once in each row and column. A \emph{Latin square} is a square Latin rectangle. Let $L$ be a Latin square with symbol set $S$. We will index the rows and columns of $L$ by $S$ and we will denote the symbol in row $i$ and column $j$ of $L$ by $L_{i, j}$. 
	
	Let $\F_q$ denote the finite field with $q$ elements. Let $\mathcal{R}_q$ and $\mathcal{N}_q$ denote the set of quadratic residues, and quadratic non-residues of the multiplicative group $\F_q^*$, respectively. Let $\{a, b\} \subseteq \F_q$ be such that $\{ab, (a-1)(b-1)\} \subseteq \mathcal{R}_q$. We can then define a $q \times q$ Latin square $\mathcal{L}[a, b]$ by
	\[
	(\mathcal{L}[a, b])_{i, j} = \begin{cases}
		i & \text{if } j = i, \\
		i + a(j-i) & \text{if } j-i \in \R_q, \\
		i + b(j-i) & \text{if } j-i \in \N_q.
	\end{cases}
	\]
	Such squares are called \emph{quadratic Latin squares}. The condition $\{ab, (a-1)(b-1)\} \subseteq \R_q$ ensures that $\mathcal{L}[a, b]$ is a Latin square~\cite{MR3837138}. Quadratic Latin squares have previously been used to construct perfect $1$-factorisations~\cite{hon, MR4070707, MR2134185}, mutually orthogonal Latin squares~\cite{MR3837138, MR1222645}, atomic Latin squares~\cite{MR2134185}, Falconer varieties~\cite{hon}, and maximally non-associative quasigroups~\cite{MR4275825, numquad}. Quadratic Latin squares are the main focus of this paper.
	
	A \emph{Latin subrectangle} of a Latin square is a submatrix which is itself a Latin rectangle. A \emph{Latin subsquare} is a square Latin subrectangle. An \emph{intercalate} is a $2 \times 2$ Latin subsquare. A Latin square is called \emph{$N_2$} if it contains no intercalates. It is known~\cite{MR401504, MR427101, MR396298, MR1820691} that an $N_2$ Latin square of order $n$ exists if and only if $n \not\in \{2, 4\}$. Such squares are also known to be rare~\cite{MR4488316, MR1685535} and can be used to construct disjoint Steiner triple systems~\cite{MR401504}. We completely characterise when a quadratic Latin square is $N_2$.
	
	\begin{thm}\label{t:n2}
		Let $q$ be an odd prime power. The Latin square $\mathcal{L}[a, b]$ of order $q$ contains an intercalate if and only if
		\[
		(2ab-a-b)(a+b)(a-1) \in \R_q \text{ and } \{2(a+b-2)(a-1), 2a(a+b)\} \subseteq \N_q,
		\]
		or both $q \equiv 1 \bmod 4$ and $b \in \{2-a, a/(2a-1), -a\}$.
	\end{thm}
	
	Let $G$ be a graph. A \emph{$1$-factor} of $G$ is a collection $M$ of its edges such that every vertex of $G$ is incident to exactly one edge in $M$. A \emph{$1$-factorisation} of $G$ is a partition of its edges into $1$-factors. Let $\mathcal{F}$ be a $1$-factorisation of $G$. Each pair of $1$-factors in $\mathcal{F}$ induces a subgraph of $G$, which is the union of cycles of even length. We will say that $\mathcal{F}$ contains these cycles. The problem of investigating $1$-factorisations which satisfy certain conditions on their cycles has received some attention. The most notable case of this is the study of perfect $1$-factorisations. If all the cycles in $\mathcal{F}$ are Hamiltonian then $\mathcal{F}$ is called \emph{perfect}. See~\cite{MR2597247, MR2311238} for applications of perfect $1$-factorisations to computer science. We will be most interested in studying $1$-factorisations of complete graphs and complete bipartite graphs. It is known that a $1$-factorisation of $K_{2n}$ exists for all positive integers $n$ and a $1$-factorisation of $K_{n, n}$ exists for all positive integers $n$. 
	
	In $1964$, Kotzig~\cite{MR0173249} conjectured that a perfect $1$-factorisation of $K_{2n}$ exists for all positive integers $n$. Despite receiving lots of attention, this conjecture remains far from resolved. There are only three known infinite families~\cite{MR2216455, MR0173249} of perfect $1$-factorisations of complete graphs. These families prove the existence of perfect $1$-factorisations of $K_{2n}$ where $2n \in \{p+1, 2p\}$ for an odd prime $p$. Perfect $1$-factorisations of $K_{2n}$ are also known to exist for some sporadic values of $n$. See~\cite{MR4070707} for a list of these values. 
	
	It is known that a perfect $1$-factorisation of $K_{n, n}$ can only exist if $n=2$ or $n$ is odd. Laufer~\cite{MR582276} showed that if there exists a perfect $1$-factorisation of $K_{2n}$ for some positive integer $n$, then there exists a perfect $1$-factorisation of $K_{2n-1, 2n-1}$. It is thus conjectured that a perfect $1$-factorisation of $K_{n, n}$ exists for all odd $n$. This conjecture also remains far from resolved. There are eight known infinite families of perfect $1$-factorisations of complete bipartite graphs~\cite{hon, MR1899629, MR2216455, MR582276}. These families prove the existence of perfect $1$-factorisations of $K_{n, n}$ where $n \in \{p, 2p-1, p^2\}$ for an odd prime $p$. There are also known perfect $1$-factorisations of $K_{n, n}$ for some sporadic values of $n$.
	
	A contrasting problem to the construction of perfect $1$-factorisations is the construction of $1$-factorisations which contain only short cycles. H\'{a}ggkvist~\cite{hagg} asked the following question. Given a graph $G$, what is the least integer $m$ such that there is a $1$-factorisation of $G$ whose cycles are all of length at most $m$. Particular interest has been given to the case where $G$ is a complete bipartite graph. It has been conjectured that for all sufficiently large $n$ there exists a $1$-factorisation of $K_{n, n}$ whose cycles are all of length at most six. This problem has been studied in~\cite{MR3537912, MR422088, MR2433008, MR2475030, MR2071339}. The current best known result, due to Benson and Dukes~\cite{MR3537912}, is that, for each positive integer $n$, there exists a $1$-factorisation of $K_{n, n}$ whose cycles are all of length at most $182$. The current best known result for complete graphs is due to Dukes and Ling~\cite{MR2475030}. It states that for all positive integers $n$, there exists a $1$-factorisation of $K_{2n}$ whose cycles are all of length at most $1720$.
	
	Let $L$ be a Latin square with symbol set $S$ of size $n$. For each $\{i, j\} \subseteq S$ with $i \neq j$, the permutation mapping row $i$ to row $j$, denoted by $r_{i, j}$, is defined by $r_{i, j}(L_{i, k}) = L_{j, k}$ for all $k \in S$. We call such permutations \emph{row permutations} of $L$ and we call each cycle in a row permutation a \emph{row cycle} of $L$. Every row cycle of $L$ has length at least two. If every row permutation of $L$ consists of a single row cycle of length $n$ then $L$ is called \emph{row-Hamiltonian}. A row cycle of length $m$ in $L$ induces a $2 \times m$ Latin subrectangle of $L$. So $L$ is $N_2$ if and only if it contains no row cycle of length two, and $L$ is row-Hamiltonian if and only if it does not contain any $m \times k$ Latin subrectangles with $1 < m \leq k < n$. 
	
	An \emph{ordered $1$-factorisation} of a graph $G$ is a $1$-factorisation with a total ordering on its $1$-factors. Let $L$ be an $n \times n$ Latin square. There is a known method to construct an ordered $1$-factorisation $\mathcal{E}$ of $K_{n, n}$ from $L$. Furthermore, for each row cycle of length $\ell$ in $L$, there is a corresponding cycle of length $2\ell$ in $\mathcal{E}$. This construction is reversible. If $L$ satisfies some symmetry conditions, then we can also construct a $1$-factorisation $\mathcal{F}$ of $K_{n+1}$ from $L$. For every row cycle of length $\ell$ in $L$, there is a corresponding cycle in $\mathcal{F}$ which has length $2\ell$ or $\ell+1$. These constructions will be discussed in further detail in \sref{s:back}. Many authors have used Latin squares to construct $1$-factorisations of graphs, including perfect $1$-factorisations. We will denote the $1$-factorisation of $K_{n, n}$ obtained from a Latin square $L$ by $\mathcal{E}(L)$, and we will denote the $1$-factorisation of $K_{n+1}$ obtained from a suitable Latin square $L$ by $\mathcal{F}(L)$.
	
	Our second main result concerns the row cycles of quadratic Latin squares.
	
	\begin{thm}\label{t:quadneg}
		Let $P$ denote the set of all odd primes. There exists a function $f : P \to \mathbb{N}$ such that every quadratic Latin square of order $q = p^d$ contains a row cycle of length at most $p$ if $d \geq f(p)$. Furthermore, if $L = \mathcal{L}[a, b]$ is a quadratic Latin square of order $q$ with $\{a, b\} \not\subseteq \F_p \cap \N_q$ then $L$ contains a row cycle of length exactly $p$ if $d \geq f(p)$.
	\end{thm}
	
	We will prove \tref{t:quadneg} by constructing a suitable function $f$ where $f(p)$ is asymptotically equal to $p\log(16)/\log(p)$. We note that this function $f$ we construct is not minimal.
	
	The \emph{cycle structure} of a permutation is a sorted list of the lengths of its cycles. Let $L = \mathcal{L}[a, b]$ be a quadratic Latin square. The cycle structure of any row permutation of $L$ is equal to the cycle structure of the row permutation $r_{0, 1}$ of $L$ or the cycle structure of the row permutation $r_{0, 1}$ of $\mathcal{L}[b, a]$ (see \lref{l:quadrowperms}). This makes it tempting to consider quadratic Latin squares when searching for perfect $1$-factorisations or $1$-factorisations which contain only short cycles. However \tref{t:quadneg} tells us that quadratic Latin squares of order $p^d$ will not be useful for constructing perfect $1$-factorisations if $d$ is too large. It also limits the usefulness of quadratic Latin squares of order $p^d$ for constructing $1$-factorisations which contain only short cycles if $d$ is too large, with the possible exception of the squares $\mathcal{L}[a, b]$ with $\{a, b\} \subseteq \F_p \cap \N_q$. Note that such squares can only exist when $d$ is odd.
	
	An \emph{anti-perfect} $1$-factorisation of a graph is a $1$-factorisation which does not contain any Hamiltonian cycles. It is known~\cite{MR2071339} that an anti-perfect $1$-factorisation of $K_{n, n}$ exists if and only if $n \not\in \{2, 3, 5\}$. The existence question of anti-perfect $1$-factorisations of complete graphs was almost completely resolved. It is known (see e.g.~\cite{MR3954017}) that an anti-perfect $1$-factorisation of $K_{2n}$ exists if $2 < 2n \equiv 2 \bmod 6$ or $4 < 2n \equiv 4 \bmod 6$. These $1$-factorisations come from Steiner $1$-factorisations. If $2n \equiv 0 \bmod 6$ then an anti-perfect $1$-factorisation of $K_{2n}$ exists if $12 \leq 2n \leq 100$. Also, the previously mentioned result of Dukes and Ling~\cite{MR2475030} implies the existence of an anti-perfect $1$-factorisation of $K_{2n}$ whenever $2n \geq 1722$. We resolve the existence problem of anti-perfect $1$-factorisations of complete graphs. 
	
	\begin{thm}\label{t:antiperf}
		There exists an anti-perfect $1$-factorisation of $K_{2n}$ if $2n \geq 8$.	
	\end{thm}
	
	We note that all $1$-factorisations of $K_{2n}$ are perfect if $2n \leq 6$. We also note that our contribution to \tref{t:antiperf} is little more than an observation that the method of Dukes and Ling~\cite{MR2475030} can be used to prove the existence of anti-perfect $1$-factorisations of $K_{2n}$ for almost all orders. 
	
	Let $L$ be a Latin square with symbol set $S$ of size $n$. By indexing the rows and columns of $L$ by $S$ we can consider $L$ as a set of $n^2$ triples of the form $(\text{row}, \text{column}, \text{symbol}) \in S^3$. A \emph{conjugate} of $L$ is a Latin square obtained from $L$ by uniformly permuting the elements of each triple. An \emph{atomic} Latin square is a Latin square whose conjugates are all row-Hamiltonian. Such squares have been studied in~\cite{MR2216455, MR4070707, MR2024246, MR2134185, MR1670298}. We define a Latin square of order $n$ to be \emph{anti-atomic} if none of its conjugates contain a row cycle of length $n$. We prove the following theorem, which is a strengthening of Theorem $5$ of~\cite{MR2071339}.
	
	\begin{thm}\label{t:antiatom}
		An anti-atomic Latin square of order $n$ exists for all $n \not\in \{2, 3, 5\}$.
	\end{thm}
	
	\tref{t:quadneg} suggests that we could build anti-perfect $1$-factorisations and anti-atomic Latin squares using quadratic Latin squares. We can indeed achieve this for some orders. To describe our results we need the following definition. 
	
	Let $L$ and $M$ be Latin squares with symbol sets $S$ and $T$, respectively. The \emph{direct product} of $L$ and $M$, denoted by $L \times M$, is the Latin square with symbol set $S \times T$ defined by $(L \times M)_{(a, b), (x, y)} = (L_{a, x}, M_{b, y})$. We can now state our last main result.
	
	\begin{thm}\label{t:quadanti}
		Let $n \not\in \{1, 3, 5, 15\}$ be an odd integer. There exists an anti-atomic Latin square of order $n$ which is the direct product of quadratic Latin squares. If $n$ contains a prime power divisor $m \neq 3$ with $m \equiv 3 \bmod 4$ then there exists a Latin square $L$ which is the direct product of quadratic Latin squares such that the $1$-factorisation $\mathcal{F}(L)$ of $K_{n+1}$ is well-defined and anti-perfect.
	\end{thm}
	
	\tref{t:quadanti} implies that we can also construct anti-perfect $1$-factorisations of complete bipartite graphs using direct products of quadratic Latin squares.
	
	The structure of this paper is as follows. In \sref{s:back} we study the relationship between Latin squares and $1$-factorisations in more depth. In \sref{ss:quadcyc} we develop a general method to study the row cycles of quadratic Latin squares. We will then apply these methods in \sref{ss:n2} to characterise quadratic Latin squares which contain row cycles of length two. This will allow us to prove \tref{t:n2}. In \sref{ss:p-cycles} we will prove \tref{t:quadneg}, and in \sref{ss:antiperf} we will prove \tref{t:antiperf}, \tref{t:antiatom} and \tref{t:quadanti}. In \sref{s:conc} we mention how \tref{t:n2} can be used to construct $N_2$ Latin squares of any odd order, and we discuss the usefulness of quadratic Latin squares for constructing $1$-factorisations which contain only short cycles.
	
	\section{Background}\label{s:back}
	
	Let $L$ be a Latin square with symbol set $S$ of size $n$. Let $\text{Sym}(S)$ denote the group of permutations of $S$. For $\{\sigma_1, \sigma_2, \sigma_3\} \subseteq \text{Sym}(S)$ we can define a Latin square $L(\sigma_1, \sigma_2, \sigma_3)$ which consists of the triples $(\sigma_1(r), \sigma_2(c), \sigma_3(s))$ for each triple $(r, c, s)$ of $L$. We say that a Latin square is \emph{isotopic} to $L$ if it is $L(\sigma_1, \sigma_2, \sigma_3)$ for some $\{\sigma_1, \sigma_2, \sigma_3\} \subseteq \text{Sym}(S)$. We say that a Latin square is \emph{isomorphic} to $L$ if it is $L(\sigma, \sigma, \sigma)$ for a permutation $\sigma \in \text{Sym}(S)$. Isotopy preserves the lengths of row cycles of a Latin square. We label each conjugate of $L$ by a $1$-line permutation which gives the order of the coordinates of the conjugate, relative to the order of the coordinates of the original square. So the $(1, 2, 3)$-conjugate of $L$ is itself and the $(2, 1, 3)$-conjugate is the matrix transpose of $L$. If $L$ is equal to its $(1, 3, 2)$-conjugate then $L$ is called \emph{involutory}. If $L_{i, i} = i$ for all $i \in S$ then $L$ is called \emph{idempotent}.
	
	We now describe the method mentioned in \sref{s:intro} which can be used to construct an ordered $1$-factorisation of $K_{n, n}$ from an $n \times n$ Latin square. Let $L$ be a Latin square with symbol set $S$ of size $n$. Label the vertices of $K_{n, n}$ by $S \times \{c, s\}$ where $(x_1, y_1)$ is adjacent to $(x_2, y_2)$ if and only if $y_1 \neq y_2$. For each $i \in S$ we construct a $1$-factor $e_i$ of $K_{n, n}$ from row $i$ of $L$ as follows. For each $j \in S$ add the edge $\{(j, c), (k, s)\}$ to $e_i$ where $L_{i, j} = k$. Then the set $\mathcal{E}(L) = \{e_i : i \in S\}$ is an ordered $1$-factorisation of $K_{n, n}$, where the order on the $1$-factors comes from the order of the rows of $L$. Furthermore, if the row permutation $r_{i, j}$ of $L$ contains a cycle of length $\ell$ then the subgraph of $K_{n, n}$ induced by the $1$-factors $e_i$ and $e_j$ contains a cycle of length $2\ell$. In particular, $L$ is row-Hamiltonian if and only if $\E(L)$ is perfect. This construction is reversible, and so every ordered $1$-factorisation of $K_{n, n}$ can be written as $\mathcal{E}(L')$ for some Latin square $L'$ of order $n$. For a more detailed description of this construction see~\cite{MR2130738}. The infinite families of perfect $1$-factorisations of complete bipartite graphs from~\cite{hon, MR1899629} were constructed using row-Hamiltonian Latin squares. In fact, the family of row-Hamiltonian Latin squares constructed in~\cite{hon} is the family of quadratic Latin squares $\mathcal{L}[-1, 2]$ of prime order $p$ with $p \equiv 1 \bmod 8$ or $p \equiv 3 \bmod 8$.
	
	As mentioned in \sref{s:intro}, if a Latin square $L$ of order $n$ satisfies some symmetry conditions then we can construct a $1$-factorisation of $K_{n+1}$ from $L$. Those symmetry conditions are that $L$ must be idempotent and involutory. We will briefly outline the construction now. For a more detailed description see~\cite{MR2130738}. Let $L$ be an idempotent, involutory Latin square with symbol set $S$ of size $n$. Let $v$ be any symbol which is not in $S$. Label the vertices of $K_{n+1}$ by $S \cup \{v\}$. For each $i \in S$ we construct a $1$-factor $f_i$ of $K_{n+1}$ from row $i$ of $L$ as follows. Add the edge $\{i, v\}$ to $f_i$, and for each $j \in S \setminus \{i\}$ add the edge $\{j, k\}$ to $f_i$ where $L_{i, j} = k$. The $1$-factor $f_i$ is well defined because $L$ is idempotent and involutory. Then $\mathcal{F}(L) = \{f_i : i \in S\}$ is an ordered $1$-factorisation of $K_{n+1}$, where the order on the $1$-factors comes from the order of the rows of $L$. Before describing the relationship between the row cycles of $L$ and the cycles in $\mathcal{F}(L)$ we will need the following lemma.
	
	\begin{lem}\label{l:lsymcyc}
		Let $L$ be an idempotent, involutory Latin square with symbol set $S$. Let $i$ and $j$ be distinct elements of $S$ and let $r = r_{i, j}$. The cycle of $r$ containing $i$ can be written as $(i, x_1, x_2, \ldots, x_k, j, y_1, y_2, \ldots, y_{k-1})$. If $r$ contains the cycle $(x_0, x_1, \ldots, x_k)$ then it also contains the cycle $(L_{i, x_k}, L_{i, x_{k-1}}, \ldots, L_{i, x_0})$. Furthermore these cycles coincide if and only if $x_\ell = i$ for some $\ell \in \{0, 1, \ldots, k\}$.
	\end{lem}
	\begin{proof}
		Throughout the proof let $X = (x_0, x_1, \ldots, x_k)$ be a cycle of $r$. We first prove that $r(L_{i, x_\ell}) = L_{i, x_{\ell-1}}$ for any $\ell \in \{0, 1, \ldots, k\}$ (where we take $\ell-1$ modulo $k+1$). Write $x_{\ell-1} = L_{i, a}$ for some $a \in S$. Then $x_\ell = r(x_{\ell-1}) = r(L_{i, a}) = L_{j, a}$. So $r(L_{i, x_\ell}) = L_{j, x_\ell} = a = L_{i, x_{\ell-1}}$ because $L$ is involutory. Therefore $r$ contains the cycle $(L_{i, x_k}, L_{i, x_{k-1}}, \ldots, L_{i, x_0})$.
		
		Suppose, for the moment, that $X$ contains $i$. Since $L$ is idempotent we know that $X$ and $(L_{i, x_k}, L_{i, x_{k-1}}, \ldots, L_{i, x_0})$ must coincide. Without loss of generality assume that $x_0 = i$. We will show that $L_{i, x_{k/2}} = j$. If $k$ is odd then we must have $x_{(k+1)/2} = L_{i, x_{(k+1)/2}}$, which is impossible since $L$ is idempotent. Therefore $k$ is even and $X$ can be written as
		\[
		(i, x_1, x_2, \ldots, x_{k/2}, L_{i, x_{k/2}}, \ldots, L_{i, x_2}, L_{i, x_1}).
		\]
		In particular we must have $r(x_{k/2}) = L_{i, x_{k/2}}$. Write $x_{k/2} = L_{i, b}$ for some $b \in S$. Then because $L$ is involutory we have that $b = L_{i, x_{k/2}} = r(x_{k/2}) = L_{j, b}$. But $L$ is idempotent, hence we must have $b = j$ and therefore $L_{i, x_{k/2}} = j$.
		
		Now suppose that $X$ is equal to the cycle $(L_{i, x_k}, L_{i, x_{k-1}}, \ldots, L_{i, x_0})$. We will show that $X$ must contain $i$. We can write $x_0 = L_{i, x_\ell}$ for some $\ell \in \{0, 1, \ldots, k\}$. Then we also have $x_m = L_{i, x_{\ell-m}}$ for each $m \in \{0, 1, \ldots, k\}$ where $\ell-m$ is taken modulo $k+1$. If $\ell$ is even then taking $m=\ell/2$ we see that $x_m = L_{i, x_m}$ which implies that $x_m=i$. If $m$ is odd then taking $m=(\ell+1)/2$ we see that $x_m = r(x_{m-1}) = r(L_{i, x_m}) = L_{j, x_m}$ which implies that $x_m=j$. Either way, $X$ must contain $i$.
	\end{proof}
	
	We can now describe the relationship between the row cycles of $L$ and the cycles in $\mathcal{F}(L)$. Let $r = r_{i, j}$ be a row permutation of $L$ and let $(i, x_1, x_2, \ldots, x_k, j, y_1, y_2, \ldots, y_{k-1})$ be the cycle of $r$ containing $i$ and $j$. Then $\mathcal{F}(L)$ contains the cycle $(v, i, x_1, y_{k-1}, x_2, y_{k-2}, \ldots, x_k, j)$. Let $(y_0, y_1, \ldots, y_k)$ be a cycle of $r$ which does not contain $i$, so that $(L_{i, y_k}, L_{i, y_{k-1}}, \ldots, L_{i, y_0})$ is also a cycle of $r$. Then $\mathcal{F}(L)$ contains the cycle $(y_0, L_{i, y_0}, y_1, L_{i, y_{1}}, \ldots, y_k, L_{i, y_k})$. In particular, $L$ is row-Hamiltonian if and only if $\mathcal{F}(L)$ is perfect.	
	
	\section{Row cycles of quadratic Latin squares}\label{ss:quadcyc}
	
	In this section we develop a method to investigate row cycles of quadratic Latin squares. The following result will be used frequently, and it is one of our primary motivators for studying quadratic Latin squares (see e.g.~\cite{hon}).
	
	\begin{lem}\label{l:quadrowperms}
		Let $q$ be an odd prime power and let $\{a, b\} \subseteq \F_q$ be such that $\{ab, (a-1)(b-1)\} \subseteq \mathcal{R}_q$. 
		\begin{enumerate}[(i)]
			\item If $q \equiv 3 \bmod 4$ then every row permutation of the Latin square $\mathcal{L}[a, b]$ has the same cycle structure as the row permutation $r_{0,1}$ of $\mathcal{L}[a, b]$.
			\item If $q \equiv 1 \bmod 4$ then every row permutation of the Latin square $\mathcal{L}[a, b]$ has the same cycle structure as either the row permutation $r_{0,1}$ of $\mathcal{L}[a, b]$ or the row permutation $r_{0,1}$ of $\mathcal{L}[b, a]$.
		\end{enumerate}
	\end{lem} 
	
	Therefore, to investigate the row cycles of quadratic Latin squares it suffices to consider only the row permutations mapping row $0$ to row $1$.
	
	Throughout this section let $q$ be an odd prime power and let $c \in \{2, 3, \ldots, q\}$. We call a pair $(a, b) \in \F_q^2$ \emph{valid} if $\{ab, (a-1)(b-1)\} \subseteq \mathcal{R}_q$. It is known~\cite{MR1222645} that the number of valid pairs in $\F_q^2$ is $(q-3)(q-5)/4+q-2$. Denote the row permutation $r_{0, 1}$ of a quadratic Latin square $\mathcal{L}[a, b]$ by $\alpha[a, b]$ and define the set
	\[
	\G = \left\{\alpha[a, b] : (a, b) \in \F_q^2 \text{ is valid}\right\}.
	\]
	For a valid pair $(a, b) \in \F_q^2$ define the permutation $\varphi[a, b]$ by
	\[
	\varphi[a, b](x) = \begin{cases}
		0 & \text{if } x = 0, \\
		ax & \text{if } x \in \R_q, \\
		bx & \text{if } x \in \N_q.
	\end{cases}
	\]
	Then $(\mathcal{L}[a, b])_{i, j} = i + \varphi[a, b](j-i)$ for all $\{i, j\} \subseteq \F_q$. Let $\alpha = \alpha[a, b]$ and let $\varphi = \varphi[a, b]$. Then $\alpha$ is defined by
	\[
	\alpha(j) = \varphi(\varphi^{-1}(j)-1)+1.
	\]
	A straightforward computation shows that $\varphi^{-1} = \varphi[a^{-1}, b^{-1}]$ if $a \in \R_q$ and $\varphi^{-1} = \varphi[b^{-1}, a^{-1}]$ if $a \in \N_q$.
	
	We now introduce some tools which can be used to investigate the cycles of a permutation $\alpha \in \G$.
	We will call a cycle of length $k$ in a permutation a $k$-cycle. For a sequence $z$, we denote the $i$-th element of $z$ by $z_i$, starting from $z_0$. For a cycle $\beta$ of $\alpha$ and element $j$ in the cycle $\beta$ we will write $j \in \beta$. Let $\eta : \F_q^* \to \mathbb{C}$ denote the quadratic character, and extend $\eta$ to $\F_q$ by defining $\eta(0) = 0$. 
	
	\begin{defin}\label{d:asat}
		Let $z \in \{-1, 0, 1\}^{2c}$ and $\alpha \in \G$. Suppose that there is a $c$-cycle $\beta$ of $\alpha$ and element $j \in \beta$ such that $z_{2k} = \eta(\alpha^k(j))$ and $z_{2k+1} = \eta(\varphi^{-1}(\alpha^k(j))-1)$ for each $k \in \{0, 1, 2, \ldots, c-1\}$. Then we say that $\alpha$ \emph{satisfies} $z$ with cycle $\beta$ and element $j \in \beta$.
	\end{defin}
	
	We will sometimes simply say that $\alpha$ satisfies $z$ or that $\alpha$ satisfies $z$ with element $j \in \F_q$. Let $\alpha \in \G$. Suppose that $\alpha$ satisfies a sequence $z \in \{-1, 0, 1\}^{2c}$ with $z_k = 0$ for some $k \in \{0, 1, 2, \ldots, 2c-1\}$. Then either $\alpha^m(j) = 0$ or $\varphi^{-1}(\alpha^m(j))-1=0$ for some $m \in \{0, 1, 2, \ldots, c-1\}$. The first case implies that $0 \in \beta$ and the second implies that $\beta$ contains $\alpha^m(j) = \varphi(1) = a$. We will let the cycles of $\alpha$ containing $0$ and $a$ be denoted by $\alpha_0$ and $\alpha_a$, respectively. We will deal with these cycles separately, hence we will mostly be concerned with sequences $z \in \{-1, 1\}^{2c}$. 
	For a positive integer $i$ and sequence $z \in \{-1, 1\}^{2c}$ let $z^i$ denote the sequence obtained by cyclically rotating $z$ by $i$ positions. That is, $z^i_{k} = z_{k+i \bmod 2c}$. We note the following simple observation.
	
	\begin{lem}\label{l:cycshift}
		Let $\alpha \in \G$. If $\alpha$ satisfies $z \in \{-1, 1\}^{2c}$ then $\alpha$ satisfies $z^{2i}$ for all $i \in \{0, 1, 2, \ldots, c-1\}$.
	\end{lem}
	\begin{proof}
		Suppose that $\alpha$ satisfies $z$ with cycle $\beta$ and element $j \in \beta$. It is simple to verify, using \dref{d:asat}, that $\alpha$ satisfies $z^{2i}$ with cycle $\beta$ and element $\alpha^i(j) \in \beta$.
	\end{proof}	
	
	We will need the following notation to deal with sequences $z \in \{-1, 1\}^{2c}$.
	
	\begin{defin}
		Let $\{i, j\} \subseteq \{0, 1, 2, \ldots, 2c-1\}$ with $i \leq j$. For a sequence $z \in \{-1, 1\}^{2c}$ we define
		\[
		\begin{aligned}
			& e^+(i, j) = |\{i \leq k \leq j : k \text{ is even and } z_k = 1\}|, \\
			& o^+(i, j) = |\{i \leq k \leq j : k \text{ is odd and } z_k = 1\}|, \\
			& e^-(i, j) = |\{i \leq k \leq j : k \text{ is even and } z_k = -1\}|, \\
			& o^-(i, j) = |\{i \leq k \leq j : k \text{ is odd and } z_k = -1\}|.
		\end{aligned}
		\]
		Also define
		\[
		\begin{aligned}
			&u^+(i, j) = o^+(i, j) - e^+(i, j), \\
			&u^-(i, j) = o^-(i, j) - e^-(i, j), \\
			&v^+(i, j) = o^+(i, j) - e^-(i, j), \\
			&v^-(i, j) = o^-(i, j) - e^+(i, j).
		\end{aligned}
		\]
	\end{defin}
	
	For $i > j$ we define $u^+(i, j) = u^-(i, j) = v^+(i, j) = v^-(i, j) = 0$. We note that the values of $u^+$, $u^-$, $v^+$ and $v^-$ implicitly depend on the choice of sequence $z \in \{-1, 1\}^{2c}$. We now prove a result concerning how permutations in $\G$ act on elements of $\F_q$. We will need to consider the cases $a \in \mathcal{R}_q$ and $a \in \mathcal{N}_q$ separately. We will repeatedly use the simple property that $u^+(i, j) + u^+(j+1, k) = u^+(i, k)$ for any $i \leq j \leq k$. The same holds when replacing $u^+$ by $u^-$, $v^+$ or $v^-$.
	
	\begin{lem}\label{l:junalphr}
		Let $\alpha = \alpha[a, b] \in \Gamma$ with $a \in \mathcal{R}_q$. Let $\varphi = \varphi[a, b]$ and $z \in \{-1, 1\}^{2c}$. Suppose that $\alpha$ satisfies $z$ with element $j \in \F_q$. Then for all $m \in \{0, 1, 2, \ldots, c\}$,
		\begin{equation}\label{e:junalph}
			\alpha^m(j) = a^{u^+(0, 2m-1)}b^{u^-(0, 2m-1)}j + \sum_{k=1}^{2m} (-1)^{k}a^{u^+(k, 2m-1)}b^{u^-(k, 2m-1)}
		\end{equation}
		and
		\begin{equation}\label{e:junalphvar}
			\varphi^{-1}(\alpha^m(j))-1 = a^{u^+(0, 2m)}b^{u^-(0, 2m)}j + \sum_{k=1}^{2m+1} (-1)^{k}a^{u^+(k, 2m)}b^{u^-(k, 2m)}.
		\end{equation}
	\end{lem}
	\begin{proof}
		We will prove the claim by induction on $m$. If $m=0$ then \eref{e:junalph} simply states that $\alpha^m(j) = j$, which is true. Since $\alpha$ satisfies $z$ we know that $\eta(j) = z_0$. Hence $\varphi^{-1}(j)-1 = a^{-e^+(0, 0)}b^{-e^-(0, 0)}j-1$, which agrees with \eref{e:junalphvar}. Now suppose that~\eref{e:junalph} and \eref{e:junalphvar} hold for some $m \geq 0$. Then
		\[
		\begin{aligned}
			\alpha^{m+1}(j) &= \varphi(\varphi^{-1}(\alpha^{m}(j))-1)+1 \\
			&= a^{o^+(2m+1, 2m+1)}b^{o^-(2m+1, 2m+1)}(\varphi^{-1}(\alpha^{m}(j))-1)+1 \\
			&= a^{u^+(2m+1, 2m+1)}b^{u^-(2m+1, 2m+1)}\left(a^{u^+(0, 2m)}b^{u^-(0, 2m)}j + \sum_{k=1}^{2m+1} (-1)^{k}a^{u^+(k, 2m)}b^{u^-(k, 2m)}\right)+1 \\
			&= a^{u^+(0, 2m+1)}b^{u^-(0, 2m+1)}j + \left(\sum_{k=1}^{2m+1} (-1)^{k}a^{u^+(k, 2m+1)}b^{u^-(k, 2m+1)}\right) + 1 \\
			&= a^{u^+(0, 2m+1)}b^{u^-(0, 2m+1)}j + \sum_{k=1}^{2m+2} (-1)^{k}a^{u^+(k, 2m+1)}b^{u^-(k, 2m+1)},
		\end{aligned}
		\]
		which agrees with~\eref{e:junalph}. Using this we have that
		\[
		\begin{aligned}
			\varphi^{-1}(\alpha^{m+1}(j))-1 &= a^{-e^+(2m+2, 2m+2)}b^{-e^-(2m+2, 2m+2)}\alpha^{m+1}(j)-1 \\
			&= a^{u^+(2m+2, 2m+2)}b^{u^-(2m+2, 2m+2)}\bigg(a^{u^+(0, 2m+1)}b^{u^-(0, 2m+1)}j + \\
			&\hspace{5mm} \sum_{k=1}^{2m+2} (-1)^{k}a^{u^+(k, 2m+1)}b^{u^-(k, 2m+1)}\bigg)-1 \\
			&= a^{u^+(0, 2m+2)}b^{u^-(0, 2m+2)}j + \left(\sum_{k=1}^{2m+2} (-1)^{k}a^{u^+(k, 2m+2)}b^{u^-(k, 2m+2)}\right) - 1 \\
			&= a^{u^+(0, 2m+2)}b^{u^-(0, 2m+2)}j + \sum_{k=1}^{2m+3} (-1)^{k}a^{u^+(k, 2m+2)}b^{u^-(k, 2m+2)},
		\end{aligned}
		\]
		which agrees with~\eref{e:junalphvar} and so the lemma follows by induction.
	\end{proof}
	
	Using analogous arguments we can prove the following result.
	
	\begin{lem}\label{l:junalphnr}
		Let $\alpha = \alpha[a, b] \in \Gamma$ with $a \in \mathcal{N}_q$. Let $\varphi = \varphi[a, b]$ and $z \in \{-1, 1\}^{2c}$. Suppose that $\alpha$ satisfies $z$ with element $j \in \F_q$. Then for all $m \in \{0, 1, 2, \ldots, c\}$,
		\begin{equation*}
			\alpha^m(j) = a^{v^+(0, 2m-1)}b^{v^-(0, 2m-1)}j + \sum_{k=1}^{2m} (-1)^{k}a^{v^+(k, 2m-1)}b^{v^-(k, 2m-1)}
		\end{equation*}
		and
		\[
		\varphi^{-1}(\alpha^m(j))-1 = a^{v^+(0, 2m)}b^{v^-(0, 2m)}j + \sum_{k=1}^{2m+1} (-1)^{k}a^{v^+(k, 2m)}b^{v^-(k, 2m)}.
		\]
	\end{lem}
	
	Let $\alpha = \alpha[a, b] \in \Gamma$. Suppose that $a \in \mathcal{R}_q$ and consider \lref{l:junalphr}. Setting $m=c$ in \eref{e:junalph} we see that
	\[
	j = a^{u^+(0, 2c-1)}b^{u^-(0, 2c-1)}j + \sum_{k=1}^{2c} (-1)^{k}a^{u^+(k, 2c-1)}b^{u^-(k, 2c-1)}.
	\]
	In order to investigate this equation we need to distinguish two cases, depending on whether or not $a^{u^+(0, 2c-1)}b^{u^-(0, 2c-1)}$ is equal to $1$. We also need to make the analogous case distinction when $a \in \mathcal{N}_q$. 
	
	Recall that an $m$-th root of unity in $\F_q$ is an element $x$ such that $x^m = 1$. 
	We will say that all non-zero elements of $\F_q$ are $0$-th roots of unity. If $m$ is a negative integer then we will say that $x$ is an $m$-th root of unity if $x^{-1}$ is a $(-m)$-th root of unity. For $\alpha = \alpha[a, b] \in \G$ and $z \in \{-1, 1\}^{2c}$ define
	\[
	t(z, \alpha) = \begin{cases}
		u^+(0, 2c-1) & \text{if } a \in \mathcal{R}_q, \\
		v^+(0, 2c-1) & \text{if } a \in \mathcal{N}_q.
	\end{cases}
	\]
	We note that $e^-(0, 2c-1) = c-e^+(0, 2c-1)$ and $o^-(0, 2c-1) = c-o^+(0, 2c-1)$. Hence
	\[
	u^-(0, 2c-1) = c-e^+(0, 2c-1) - (c-o^+(0, 2c-1)) = o^+(0, 2c-1) - e^+(0, 2c-1) = -u^+(0, 2c-1).
	\]
	So $a^{u^+(0, 2c-1)}b^{u^-(0, 2c-1)} = (ab^{-1})^{t(z, \alpha)}$ if $a \in \R_q$ and $a^{v^+(0, 2c-1)}b^{v^-(0, 2c-1)} = (ab^{-1})^{t(z, \alpha)}$ if $a \in \N_q$.
	We therefore make the following definition.
	
	\begin{defin}
		Let $z \in \{-1, 1\}^{2c}$ and $\alpha = \alpha[a, b] \in \G$. We say that the pair $(z, \alpha)$ is of \emph{Type One} if $ab^{-1}$ is not a $t(z, \alpha)$-th root of unity in $\F_q$. Otherwise we say that $(z, \alpha)$ is of \emph{Type Two}.
	\end{defin}
	
	Fix a permutation $\alpha \in \G$. We will say that a sequence $z \in \{-1, 1\}^{2c}$ is a Type One sequence or Type Two sequence according to whether the pair $(z, \alpha)$ is of Type One or Type Two. Let $\beta \not\in \{\alpha_0, \alpha_a\}$ be a cycle of $\alpha$ and let $j \in \beta$. Using \dref{d:asat} we can associate a sequence $z \in \{-1, 1\}^{2c}$ to the cycle $\beta$ and element $j \in 
	\beta$. Furthermore, by \lref{l:cycshift} we know that changing the element $j$ of $\beta$ simply cyclically rotates the sequence $z$ by an even number of positions. It is clear that $(z, \alpha)$ is of Type One if and only if $(z^{2i}, \alpha)$ is of Type One, for all $i \in \{0, 1, 2, \ldots, c-1\}$. Thus we define $\beta$ to be a Type One cycle if $(z, \alpha)$ is of Type One, and we define $\beta$ to be a Type Two cycle otherwise. 
	
	Our goal in this section is to develop a method to investigate the cycles of a permutation $\alpha \in \G$. To do this we will study Type One cycles, Type Two cycles, and the cycles $\alpha_0$ and $\alpha_a$, separately.
	
	\subsection{Type One cycles}\label{sss:t1}
	
	The goal of this subsection is to prove necessary and sufficient conditions for a permutation in $\G$ to contain a Type One cycle of length $c$. Let $k$ be a positive integer and $\{x, y\} \subseteq \{-1, 1\}^{k}$. We define the \emph{concatenation} of $x$ and $y$, denoted by $x \oplus y$, to be the sequence $(x_0, x_1, \ldots, x_{k-1}, y_0, y_1, \ldots, y_{k-1}) \in \{-1, 1\}^{2k}$.
	
	\begin{defin}
		Let $k$ be a positive integer and $z \in \{-1, 1\}^{2k}$. We call $z$ \emph{even periodic} if we can write $z = \bigoplus_{i=1}^{k/d} y$ for some proper divisor $d$ of $k$ and some $y \in \{-1, 1\}^{2d}$.
	\end{defin}
	
	Let $z \in \{-1, 1\}^{2c}$ be even periodic so that we can write $z = \bigoplus_{k=1}^{c/d} y$ for some positive integer $d$ and some $y \in \{-1, 1\}^{2d}$. Observe the following simple consequence of the even periodicity of $z$.
	\[
	u^+(k, 2c-1) = \left(\frac cd - \left\lceil \frac{k+1}{2d} \right\rceil \right) u^+(0, 2d-1) + u^+(k \bmod 2d, 2d-1)
	\]
	for all $k \in \{0, 1, 2, \ldots, 2c-1\}$. In particular we have that $u^+(0, 2c-1) = (c/d)u^+(0, 2d-1)$. The same holds when replacing $u^+$ by $v^+$. We will now show that a permutation in $\G$ cannot satisfy an even periodic sequence of Type One. 
	
	\begin{lem}\label{l:eperiod}
		Let $\alpha \in \G$ and $z \in \{-1, 1\}^{2c}$ be an even periodic sequence. 
		If $\alpha$ satisfies $z$ then $(z, \alpha)$ is of Type Two.
	\end{lem}
	\begin{proof}
		Write $\alpha = \alpha[a, b]$ for some valid pair $(a, b) \in \F_q^2$. Write $z = \bigoplus_{k=1}^{c/d} y$ for some proper divisor $d$ of $c$ and some $y \in \{-1, 1\}^{2d}$. 
		Assume that $\alpha$ satisfies $z$ with cycle $\beta$ and element $j \in \beta$. Suppose that $a \in \mathcal{R}_q$. From \lref{l:junalphr} we know that
		\begin{equation}\label{e:alphnj}
			\begin{aligned}
				\alpha^c(j) &= a^{u^+(0, 2c-1)}b^{u^-(0, 2c-1)}j + \sum_{k=1}^{2c} (-1)^{k}a^{u^+(k, 2c-1)}b^{u^-(k, 2c-1)} \\
				&= (a^{u^+(0, 2d-1)}b^{u^-(0, 2d-1)})^{c/d}j + \\
				&\hspace{5mm} \sum_{k=1}^{2c} (-1)^{k}\left((a^{u^+(0, 2d-1)}b^{u^-(0, 2d-1)})^{c/d-\lceil (k+1)/(2d) \rceil} \cdot a^{u^+(k \bmod 2d, 2d-1)}b^{u^-(k \bmod 2d, 2d-1)}\right) \\
				&= (a^{u^+(0, 2d-1)}b^{u^-(0, 2d-1)})^{c/d}j + 
				\sum_{k=1}^{2d}(-1)^k \sum_{i=0}^{c/d-1} (a^{u^+(0, 2d-1)}b^{u^-(0, 2d-1)})^ia^{u^+(k, 2d-1)}b^{u^-(k, 2d-1)} \\
				&= (a^{u^+(0, 2d-1)}b^{u^-(0, 2d-1)})^{c/d}j + 
				\sum_{i=0}^{c/d-1} (a^{u^+(0, 2d-1)}b^{u^-(0, 2d-1)})^i \sum_{k=1}^{2d}(-1)^ka^{u^+(k, 2d-1)}b^{u^-(k, 2d-1)}.
			\end{aligned}
		\end{equation}
		Now suppose, for a contradiction, that $z$ is a Type One sequence. So $a^{u^+(0, 2c-1)}b^{u^-(0, 2c-1)} \neq 1$, hence $a^{u^+(0, 2d-1)}b^{u^-(0, 2d-1)} \neq 1$ also. Thus we can write
		\[
		\sum_{i=0}^{c/d-1} (a^{u^+(0, 2d-1)}b^{u^-(0, 2d-1)})^i = \frac{(1-(a^{u^+(0, 2d-1)}b^{u^-(0, 2d-1)})^{c/d})}{(1-a^{u^+(0, 2d-1)}b^{u^-(0, 2d-1)})}.
		\]
		Substituting this into \eref{e:alphnj} we have that
		\[
		\alpha^c(j) = (a^{u^+(0, 2d-1)}b^{u^-(0, 2d-1)})^{c/d}j + \frac{(1-(a^{u^+(0, 2d-1)}b^{u^-(0, 2d-1)})^{c/d})}{(1-a^{u^+(0, 2d-1)}b^{u^-(0, 2d-1)})}\sum_{k=1}^{2d}(-1)^ka^{u^+(k, 2d-1)}b^{u^-(k, 2d-1)}.
		\]
		Since $\alpha^c(j) = j$ we obtain
		\[
		j = \frac1{1-a^{u^+(0, 2d-1)}b^{u^-(0, 2d-1)}}\sum_{k=1}^{2d}(-1)^ka^{u^+(k, 2d-1)}b^{u^-(k, 2d-1)}.
		\]
		By clearing the denominator we obtain
		\[
		j = a^{u^+(0, 2d-1)}b^{u^-(0, 2d-1)}j + \sum_{k=1}^{2d}(-1)^ka^{u^+(k, 2d-1)}b^{u^-(k, 2d-1)} = \alpha^d(j)
		\]
		from \eref{e:junalph}. This contradicts the fact that $\beta$ is a $c$-cycle. The case where $a \in \mathcal{N}_q$ can be handled using analogous arguments.
	\end{proof}
	
	Let $\alpha \in \G$. Define
	\[
	X_{c, \alpha} = \{z \in \{-1, 1\}^{2c} : z \text{ is not even periodic and } (z, \alpha) \text{ is of Type One}\}.
	\]
	
	Define an equivalence relation $\sim$ on $X_{c, \alpha}$ by $z \sim y$ if and only if $z = y^{2i}$ for some $i \in \{0, 1, \ldots, c-1\}$. It is simple to verify that $\sim$ is indeed an equivalence relation on this set. For notational convenience we will identify an equivalence class of $X_{c, \alpha} / {\sim}$ with an element of that equivalence class.  
	By combining \lref{l:cycshift} and \lref{l:eperiod} we have the following result.
	
	\begin{lem}\label{l:cycsseqs}
		A permutation $\alpha \in \G$ contains a Type One $c$-cycle if and only if it satisfies a sequence in $X_{c, \alpha} / {\sim}$. 
	\end{lem}
	
	For a positive integer $m$ let $\Gamma_m$ be the subset of $\Gamma$ consisting of elements $\alpha[a, b]$ where $ab^{-1}$ is not a $k$-th root of unity, for any $k \in \{1, 2, \ldots, m\}$. We note that if $\alpha \in \G_c$ then the set $X_{c, \alpha}$ depends only on whether $a \in \mathcal{R}_q$ or $a \in \mathcal{N}_q$. Hence we use the term $X_{c, 1}$ to denote the set $X_{c, \alpha}$ for some $\alpha = \alpha[a, b] \in \G_c$ with $a \in \mathcal{R}_q$, and we write $X_{c, 2}$ to denote the set $X_{c, \alpha}$ for some $\alpha = \alpha[a, b] \in \G_c$ with $a \in \mathcal{N}_q$. The number of elements in the sets $X_{c, 1} / {\sim}$ and $X_{c, 2} / {\sim}$ is related to the number of Lyndon words of length $c$ over an alphabet of size four.
	
	We will now find necessary and sufficient conditions for a permutation $\alpha \in \G$ to satisfy a sequence in $X_{c, \alpha} / {\sim}$. Let $z \in \{-1, 1\}^{2c}$ and define the bivariate Laurent polynomial $F_{0, z}$ over $\F_q$ by
	\[
	F_{0, z}(x, y) = (1-(xy^{-1})^{u^+(0, 2c-1)})\sum_{k=1}^{2c}(-1)^kx^{u^+(k, 2c-1)}y^{u^-(k, 2c-1)}.
	\]
	Then for $i \in \{1, 2, \ldots, 2c-1\}$ define
	\[
	F_{i, z}(x, y) = x^{u^+(0, i-1)}y^{u^-(0, i-1)}F_{0, z}(x, y) + (1-(xy^{-1})^{u^+(0, 2c-1)})^2\sum_{k=1}^{i} (-1)^kx^{u^+(k, i-1)}y^{u^-(k, i-1)}.
	\]
	Also define the bivariate Laurent polynomial
	\[
	G_{0, z}(x, y) = (1-(xy^{-1})^{v^+(0, 2c-1)})\sum_{k=1}^{2c}(-1)^kx^{v^+(k, 2c-1)}y^{v^-(k, 2c-1)},
	\]
	and for $i \in \{1, 2, \ldots, 2c-1\}$ define
	\[
	G_{i, z}(x, y) = x^{v^+(0, i-1)}y^{v^-(0, i-1)}G_{0, z}(x, y) + (1-(xy^{-1})^{v^+(0, 2c-1)})^2\sum_{k=1}^{i} (-1)^kx^{v^+(k, i-1)}y^{v^-(k, i-1)}.
	\]
	
	\begin{lem}\label{l:t1necsuf}
		Let $\alpha = \alpha[a, b] \in \Gamma$ and $z \in X_{c, \alpha} / {\sim}$. If $a \in \mathcal{R}_q$ then $\alpha$ satisfies $z$ if and only if $\eta(F_{i, z}(a, b)) = z_i$ for all $i \in \{0, 1, 2, \ldots, 2c-1\}$. If $a \in \mathcal{N}_q$ then $\alpha$ satisfies $z$ if and only if $\eta(G_{i, z}(a, b)) = z_i$ for all $i \in \{0, 1, 2, \ldots, 2c-1\}$.
	\end{lem}
	\begin{proof}
		We will prove the lemma in the case where $a \in \R_q$. The case where $a \in \N_q$ can be proven using analogous arguments. Suppose that $\alpha$ satisfies $z$ with element $j \in \F_q$. It follows from \lref{l:junalphr} that $F_{2i, z}(a, b) = (1-(ab^{-1})^{u^+(0, 2c-1)})^2\alpha^i(j)$ and $F_{2i+1, z}(a, b) = (1-(ab^{-1})^{u^+(0, 2c-1)})^2(\varphi^{-1}(\alpha^i(j))-1)$ for all $i \in \{0, 1, 2, \ldots, c-1\}$. Since $\alpha$ satisfies $z$ we know that $\eta(F_{i, z}(a, b)) = \eta((1-(ab^{-1})^{u^+(0, 2c-1)})^2F_{i, z}(a, b)) = z_i$ for all $i \in \{0, 1, 2, \ldots, 2c-1\}$. Now suppose that $\eta(F_{i, z}(a, b)) = z_i$ for all $i \in \{0, 1, 2, \ldots, 2c-1\}$. It is simple to verify that $\alpha$ satisfies $z$ with element $j = F_{0, z}(a, b)/(1-(ab^{-1})^{u^+(0, 2c-1)})^2$. 
	\end{proof}
	
	Combining \lref{l:cycsseqs} and \lref{l:t1necsuf} we can obtain necessary and sufficient conditions for a permutation in $\G$ to contain a Type One cycle of length $c$. We will see in \sref{ss:n2} that we can use these conditions to bound the number of permutations in $\G_c$ which contain a Type One $c$-cycle.
	
	\subsection{Type Two cycles}\label{sss:t2}
	
	In this subsection we provide necessary conditions for a permutation in $\G$ to contain a Type Two $c$-cycle. We also describe how to use these conditions to bound the number of permutations in $\G_c$ which contain a Type Two cycle of length $c$. For a permutation $\alpha \in \G$, define $Y_{c, \alpha}$ to be the set of sequences $z \in \{-1, 1\}^{2c}$ such that $(z, \alpha)$ is of Type Two. Note that for a permutation $\alpha = \alpha[a, b] \in \G_c$, the set $Y_{c, \alpha}$ depends only on whether $a \in \mathcal{R}_q$ or $a \in \mathcal{N}_q$. Therefore we will write $Y_{c, 1}$ to be $Y_{c, \alpha}$ for some $\alpha = \alpha[a, b] \in \G_c$ with $a \in \mathcal{R}_q$. Similarly we will write $Y_{c, 2}$ to be $Y_{c, \alpha}$ for some $\alpha = \alpha[a, b] \in \G_c$ with $a \in \mathcal{N}_q$. The following is a consequence of \lref{l:cycshift}.
	
	\begin{lem}\label{l:t2cycs}
		A permutation $\alpha \in \G$ contains a Type Two $c$-cycle if and only if it satisfies a sequence in $Y_{c, \alpha} / {\sim}$.
	\end{lem}
	
	Let $f(x_1, x_2, \ldots, x_k)$ be a Laurent polynomial over $\F_q$ and let $i \in \{1, 2, \ldots, k\}$. The \emph{total degree} of $f$ in $x_i$, denoted by $\deg(f, x_i)$, is the difference between the maximum power of $x_i$ in $f$, and the minimum power of $x_i$ in $f$. If $k=1$ then we say that $f$ has total degree $\deg(f, x_1)$.
	
	\begin{lem}\label{l:t2polys}
		Let $z \in \{-1, 1\}^{2c}$. There is a bivariate Laurent polynomial $g(x, y)$ over $\F_q$ with $\deg(g, x) \leq 2c$ and $\deg(g, y) \leq 2c$ such that if $\alpha = \alpha[a, b] \in \G$ satisfies $z$, $a \in \R_q$ and $(z, \alpha)$ is of Type Two then $(a, b)$ is a root of $g$.
		Similarly there is a polynomial $h(x, y)$ over $\F_q$ with $\deg(h, x) \leq 2c$ and $\deg(h, y) \leq 2c$ such that if $\alpha = \alpha[a, b] \in \G$ satisfies $z$, $a \in \N_q$ and $(z, \alpha)$ is of Type Two then $(a, b)$ is a root of $h$.
	\end{lem}
	\begin{proof}
		Let $\alpha = \alpha[a, b] \in \G$ satisfy $z$ and be such that $(z, \alpha)$ is of Type Two. First suppose that $a \in \mathcal{R}_q$. As $\alpha$ satisfies $z$ it follows from \lref{l:junalphr} that $(a, b)$ is a root of the bivariate Laurent polynomial
		\[
		g(x, y) = \sum_{k=1}^{2c} (-1)^kx^{u^+(k, 2c-1)}y^{u^-(k, 2c-1)}.
		\]
		The total degree of $g$ in $y$ is equal to the quantity $\max\{u^-(k, 2c-1) : k \in \{1, 2\ldots, 2c\}\} - \min\{u^-(k, 2c-1) : k \in \{1, 2\ldots, 2c\}\} \leq 2c$ because $u^-(k, 2c-1) \leq c$ for any $k \in \{0, 1, 2, \ldots, 2c-1\}$. Similarly $\deg(g, x) \leq 2c$. The case where $a \in \mathcal{N}_q$ can be handled using analogous arguments.
	\end{proof}
	
	We will denote the Laurent polynomials $g$ and $h$ in \lref{l:t2polys} associated to the sequence $z \in \{-1, 1\}^{2c}$ by $g_z$ and $h_z$, respectively. \lref{l:t2cycs} and \lref{l:t2polys} could be used to bound the number of permutations in $\G_c$ which contain a Type Two $c$-cycle. The number of roots of a non-zero bivariate Laurent polynomial $f(x, y)$ over $\F_q$ is bounded by $q\deg(f, y)$. If a permutation $\alpha[a, b] \in \G_c$ with $a \in \R_q$ contains a $c$-cycle then $(a, b)$ must be a root of $g_z$ for some $z \in Y_{c, 1} / {\sim}$. If $g_z$ is not the zero polynomial for any  $z \in Y_{c, 1} / {\sim}$, then we can use \lref{l:t2polys} to bound the number of permutations $\alpha[a, b] \in \G_c$ with $a \in \R_q$ which contain a Type Two $c$-cycle. Similarly, if $h_z$ is not the zero polynomial for any  $z \in Y_{c, 2} / {\sim}$ then we can bound the number of permutations $\alpha[a, b] \in \G_c$ with $a \in \N_q$ which contain a Type Two $c$-cycle. However we note that if $c$ is equal to the characteristic of $\F_q$, then there do exist sequences $z \in (Y_{c, 1} / {\sim}) \cup (Y_{c, 2} / {\sim})$ such that $g_z$ or $h_z$ is the zero polynomial. This fact will be used in \sref{ss:p-cycles}.
	
	\subsection{$\alpha_0$ and $\alpha_a$}\label{sss:a0aa}
	
	In this subsection we bound the number of permutations $\alpha \in \G$ such that $\alpha_0$ or $\alpha_a$ have length $c$.
	
	\begin{lem}\label{l:a0aapoly}
		Let $m \in \{0, 1, 2, \ldots, q-1\}$. There is a set $T_m$ containing at most $4^m$ trivariate Laurent polynomials over $\F_q$ which satisfies the following property: For every $\alpha = \alpha[a, b] \in \G$ and every $j \in \F_q$, there is some $t \in T_m$ such that $\alpha^m(j) = t(a, b, j)$. Furthermore, for each $t(x, y, z) \in T_m$ it holds that $\deg(t, x) \leq m$ and $\deg(t, y) \leq m$.
	\end{lem}
	\begin{proof}
		We will prove the claim by induction on $m$. When $m=0$, the set $T_0$ containing the polynomial $t(x, y, z)=z$ suffices. Now suppose that the claim is true for some $m \geq 0$. Let $\alpha = \alpha[a, b] \in \G$.
		By induction we know that $\alpha^{m+1}(j) = \alpha(t(a, b, j))$ for some $t \in T_m$. If $a \in \R_q$ then
		\[
		\alpha^{m+1}(j) = \begin{cases}
			t(a, b, j) - a + 1 & \text{if } \{t(a, b, j), a^{-1}t(a, b, j)-1\} \subseteq \mathcal{R}_q, \\
			a^{-1}bt(a, b, j)-b+1 & \text{if } t(a, b, j) \in \mathcal{R}_q \text{ and } a^{-1}t(a, b, j)-1 \in \mathcal{N}_q, \\
			ab^{-1}t(a, b, j) - a + 1 & \text{if } t(a, b, j) \in \mathcal{N}_q \text{ and } b^{-1}t(a, b, j)-1 \in \mathcal{R}_q, \\
			t(a, b, j) - b + 1 & \text{if } \{t(a, b, j), b^{-1}t(a, b, j)-1\} \subseteq \mathcal{N}_q.
		\end{cases}
		\]
		Similarly if $a \in \N_q$ then $\alpha^{m+1}(j) \in \{t(a, b, j) - a + 1, a^{-1}bt(a, b, j)-b+1, ab^{-1}t(a, b, j) - a + 1, t(a, b, j) - b + 1\}$.
		Define $T_{m+1} = \{t(x, y, z) - x + 1, x^{-1}yt(x, y, z)-y+1, xy^{-1}t(x, y, z) - x + 1, t(x, y, z) - y + 1 : t \in T_m\}$. By construction $\alpha^{m+1}(j) = t(a, b, j)$ for some $t \in T_{m+1}$. Also, $|T_{m+1}| \leq 4|T_m| \leq 4^{m+1}$ by induction. Furthermore, each $t \in T_{m+1}$ has been obtained from some $t' \in T_m$. The process of changing $t'$ to $t$ increases the total degree in $x$ and the total degree in $y$ by at most one. Therefore $\deg(t, x) \leq m+1$ and $\deg(t, y) \leq m+1$ for all $t \in T_{m+1}$. 
	\end{proof}
	
	We can use \lref{l:a0aapoly} to bound the number of permutations $\alpha \in \G$ such that $\alpha_0$ or $\alpha_a$ is a $c$-cycle. Let $T_c$ be the set of trivariate Laurent polynomials from \lref{l:a0aapoly}. The number of pairs $(a, b)$ which are solutions to the equation $t(x, y, 0) = 0$ for some $t \in T$ is at most $qc$. As $|T| \leq 4^c$ it follows that the number of permutations $\alpha \in \G$ with $\alpha_0$ being a $c$-cycle is at most $qc4^c$. The same conclusion holds for the number of permutations $\alpha \in \G$ such that $\alpha_a$ is of length $c$.
	
	\section{$N_2$ quadratic Latin squares}\label{ss:n2}
	
	In this section we will apply the results proven in \sref{ss:quadcyc} to investigate permutations in $\G$ which contain cycles of length two, also known as \emph{transpositions}. This will allow us to prove \tref{t:n2}. 
	Throughout this section let $q$ be an odd prime power. We will first determine when a permutation in $\G$ contains a Type One transposition. To do this, we construct the sequences in $X_{2, 1} / {\sim}$ and $X_{2, 2} / {\sim}$. We then apply \lref{l:t1necsuf} to these sequences to obtain necessary and sufficient conditions for a permutation in $\G_2$ to contain a Type One transposition. We know that the set $\G \setminus \G_2$ consists of the permutations $\alpha[a, a]$ and $\alpha[a, -a]$, which will be dealt with separately. Following the described method we obtain the following result.
	
	\begin{lem}\label{l:t1}
		The permutation $\alpha[a, b] \in \G_2$ contains a Type One transposition if and only if
		\[
		(2ab-a-b)(a+b)(a-1) \in \R_q \text{ and } \{2(a+b-2)(a-1), 2a(a+b)\} \subseteq \N_q.
		\]
	\end{lem}
	\begin{proof}
		We distinguish four cases, depending on whether $q \equiv 1 \bmod 4$ or $q \equiv 3 \bmod 4$ and whether $a \in \R_q$ or $a \in \N_q$. There are only minor differences in the arguments for these four cases so we will only prove the case where $q \equiv 3 \bmod 4$ and $a \in \R_q$. By iterating through the sequences in $X_{2, 1} / {\sim}$ and using \lref{l:t1necsuf} we can determine that a permutation $\alpha[a, b] \in \G_2$ with $a \in \R_q$ contains a Type One transposition if and only if: 
		\begin{enumerate}[$(i)$]
			\item $\{(2ab-a-b)(a-b), b(a+b)(b-1)(a-b), b(a+b-2)(a-b), 2(1-b)(a-b)\} \subseteq \N_q$,
			\item $\{(a+b)(1-b)(a-b), b(a+b-2ab)(a-b), 2a(b-1)(a-b), (2-a-b)(a-b)\} \subseteq \N_q$,
			\item $\{2b(a-1)(a-b), (2-a-b)(a-b), (1-a)(a+b)(a-b), a(a+b-2ab)(a-b)\} \subseteq \N_q$, or 
			\item $\{a(a+b-2)(a-b), 2(1-a)(a-b), (2ab-a-b)(a-b), a(a-1)(a+b)(a-b)\} \subseteq \N_q$.
		\end{enumerate} 
		Using the fact that $-1 \in \mathcal{N}_q$ and $\{a, b, (a-1)(b-1)\} \subseteq \mathcal{R}_q$ we can combine conditions $(i)$ and $(iv)$ to be
		\begin{equation}\label{e:v}
			\{(2ab-a-b)(a-b), (a+b)(a-1)(a-b), (a+b-2)(a-b), 2(1-a)(a-b)\} \subseteq \N_q.
		\end{equation}
		Similarly we can combine conditions $(ii)$ and $(iii)$ to be
		\begin{equation}\label{e:vi}
			\{(a+b)(1-a)(a-b), (a+b-2ab)(a-b), 2(a-1)(a-b), (2-a-b)(a-b)\} \subseteq \N_q.
		\end{equation}
		The lemma then follows by combining \eref{e:v} and \eref{e:vi}.
	\end{proof}
	
	We will now determine when a permutation in $\G_2$ contains a Type Two transposition. To do this, we first compute the sets $Y_{2, 1} / {\sim}$ and $Y_{2, 2} / {\sim}$. We know from \lref{l:t2polys} that if $\alpha[a, b] \in \G_2$ satisfies a sequence in one of these sets, then $(a, b)$ must be a root of some bivariate Laurent polynomial. By iterating through the sequences in $Y_{2, 1} / {\sim}$ and $Y_{2, 2} / {\sim}$ and constructing the associated Laurent polynomials we obtain the following lemma.
	\begin{lem}\label{l:n2t2}
		Let $\alpha = \alpha[a, b] \in \G_2$ and recall that $1 \not\in \{a, b\}$. If $\alpha$ contains a Type Two transposition then the pair $(a, b)$ is a solution to one of the following equations:
		\begin{enumerate}[(i)]
			\item $2-a-b=0$,
			\item $1-2b+ba^{-1}=0$.
		\end{enumerate}
	\end{lem}
	
	Checking the solutions of equations $(i)$ and $(ii)$ in \lref{l:n2t2} we obtain the following corollary.
	
	\begin{cor}\label{c:n2t2}
		Let $\alpha = \alpha[a, b] \in \G_2$ with $b \not\in \{2-a, a/(2a-1)\}$. Then $\alpha$ does not contain a Type Two transposition.
	\end{cor}
	
	We will now find conditions for a permutation $\alpha \in \G_2$ to satisfy $\alpha^2(0)=0$ or $\alpha^2(a)=a$. 
	
	\begin{lem}\label{l:n2alph0}
		Let $\alpha = \alpha[a, b] \in \G_2$ with $b \not\in \{2-a, a/(2a-1)\}$. Then $\alpha_0$ is not a transposition.
	\end{lem}
	\begin{proof}
		We will distinguish four cases, depending on whether $q \equiv 1 \bmod 4$ or $q \equiv 3 \bmod 4$, and whether $a \in \R_q$ or $a \in \N_q$. We will consider the case where $q \equiv 3 \bmod 4$ and $a \in \mathcal{R}_q$. The other cases can be dealt with using similar arguments. Since $-1 \in \mathcal{N}_q$ we have that $\alpha(0) = \varphi(\varphi^{-1}(0)-1)+1 = \varphi(-1)+1 = 1-b$. Hence
		\[
		\alpha^2(0) = \begin{cases}
			2-b-a & \text{if } \{1-b, a^{-1}-a^{-1}b-1\} \subseteq \mathcal{R}_q, \\
			a^{-1}b-a^{-1}b^2-b+1 & \text{if } 1-b \in \mathcal{R}_q \text{ and } a^{-1}-a^{-1}b-1 \in \mathcal{N}_q, \\
			ab^{-1}-2a+1 & \text{if } 1-b \in \mathcal{N}_q \text{ and } b^{-1}-2 \in \mathcal{R}_q, \\
			2-2b & \text{if } \{1-b, b^{-1}-2\} \subseteq \mathcal{N}_q. 
		\end{cases}
		\]
		Suppose that $\alpha_0$ is a transposition. If $\{1-b, a^{-1}-a^{-1}b-1\} \subseteq \R_q$ then $b=2-a$. If $1-b \in R_q$ and $a^{-1}-a^{-1}b-1 \in \mathcal{N}_q$ then $b$ is a root of the polynomial $a^{-1}x^2+x(1-a^{-1})-1 = a^{-1}(x-1)(x+a)$. As $b \neq 1$ we must have $b=-a$ and thus $\alpha \not\in \G_2$. If $1-b \in \N_q$ and $b^{-1}-2 \in \R_q$ then $b=a/(2a-1)$. Finally if $\{1-b, b^{-1}-2\} \subseteq \N_q$ then $b=1$ which is false.
	\end{proof}
	
	\begin{lem}\label{l:n2alpha}
		Let $\alpha = \alpha[a, b] \in \G_2$ with $b \not\in \{2-a, a/(2a-1)\}$. Then $\alpha_a$ is not a transposition.
	\end{lem}
	\begin{proof}
		We will first prove the claim assuming that $a \in \mathcal{R}_q$. We have that $\alpha(a) = \varphi(\varphi^{-1}(a)-1)+1 = \varphi(0)+1 = 1$. Hence
		\[
		\alpha^2(a) = \varphi(a^{-1}-1)+1 =  \begin{cases}
			2-a & \text{if } a^{-1}-1 \in \mathcal{R}_q, \\
			a^{-1}b-b+1 & \text{if } a^{-1}-1 \in \mathcal{N}_q.
		\end{cases}
		\]
		If $\alpha_a$ is a transposition then either $a=1$ or $b = (a-1)/(a^{-1}-1) = -a$, both of which are false. Using similar arguments we can show that if $a \in \N_q$ and $\alpha^2(a) = a$ then $b \in \{2-a, a/(2a-1)\}$.
	\end{proof}
	
	By combining \lref{l:t1}, \cyref{c:n2t2}, \lref{l:n2alph0} and \lref{l:n2alpha} we have completely classified when a permutation $\alpha[a, b] \in \G$ with $b \not\in \{a, -a, 2-a, a/(2a-1)\}$ contains a transposition. It is known that quadratic Latin squares of the form $\mathcal{L}[a, a]$ are isotopic to the Cayley table of the additive group $(\F_q, +)$. Therefore when $b=a$ the square $\mathcal{L}[a, b]$ does not contain a transposition. So it remains to deal with the permutations $\alpha[a, b] \in \G$ with $b \in \{-a, 2-a, a/(2a-1)\}$. If $q \equiv 3 \bmod 4$ then such permutations are not well defined. As $-1 \in \mathcal{N}_q$ we have $-a^2 \in \mathcal{N}_q$ hence $(a, -a)$ is not valid. As $(2-a-1)(a-1) = -(a-1)^2 \in \mathcal{N}_q$ the pair $(a, 2-a)$ is also not valid. Finally we note that if $(a, a/(2a-1))$ is a valid pair then we must have both $a^2/(2a-1) \in \mathcal{R}_q$ and $-(a-1)^2/(2a-1) \in \mathcal{R}_q$ and clearly both cannot be true. So to complete our classification of the permutations in $\G$ which contain a transposition we must now consider the permutations $\alpha[a, -a]$, $\alpha[a, 2-a]$ and $\alpha[a, a/(2a-1)]$ in the case where $q \equiv 1 \bmod 4$. We will in fact show that almost all of these permutations contain a transposition.
	
	\begin{lem}\label{l:2-a}
		Suppose that $q \equiv 1 \bmod 4$. Let $a \in \F_q$ with $a(2-a) \in \mathcal{R}_q$ and let $\alpha = \alpha[a, 2-a] \in \G$. If there exists some $j \in \F_q$ such that $\{aj, a^{-1}j-1\} \subseteq \R_q$ and $\{a(j-a+1), a(j-1)\} \subseteq \N_q$ then $\alpha^2(j) = j$.
	\end{lem}
	\begin{proof}
		We have that
		\[
		\begin{aligned}
			\alpha^2(j) &= \varphi(\varphi^{-1}(\varphi(\varphi^{-1}(j)-1)+1)-1)+1 \\
			&= \varphi(\varphi^{-1}(\varphi(a^{-1}j-1)+1)-1)+1 \\
			&= \varphi(\varphi^{-1}(j-a+1)-1)+1 \\
			&= \varphi((j-1)/(2-a))+1 \\
			&= j,
		\end{aligned}
		\]
		as required.
	\end{proof}
	
	Using analogous arguments we can show the following lemmas.
	
	\begin{lem}\label{l:a/2a-1}
		Suppose that $q \equiv 1 \bmod 4$. Let $a \in \F_q$ with $2a-1 \in \mathcal{R}_q$ and let $\alpha = \alpha[a, a/(2a-1)] \in \G$. If there exists some $j \in \F_q$ such that $\{aj, a(j-1)\} \subseteq \R_q$ and $\{a^{-1}j-1, a(j+a-1)\} \subseteq \N_q$ then $\alpha^2(j) = j$.
	\end{lem}
	
	\begin{lem}\label{l:-a}
		Suppose that $q \equiv 1 \bmod 4$. Let $a \in \F_q$ with $(a-1)(a+1) \in \mathcal{R}_q$ and let $\alpha = \alpha[a, -a] \in \G$. If there exists some $j \in \F_q$ such that $\{aj, a(j-a-1)\} \subseteq \R_q$ and $\{a^{-1}j-1, a(j-1)\} \subseteq \N_q$ then $\alpha^2(j) = j$.
	\end{lem}
	
	To finish the classification of permutations in $\G$ which contain a transposition we will need some tools. For convenience, if $f$ is a Laurent polynomial over $\F_q$ with a pole at $0$, then we will say that $f(0) = \infty$. We will then define $\eta(\infty) = 0$. The following~\cite{MR27006} is a version of the Weil bound.
	
	\begin{thm}\label{t:weil}
		Let $f$ be a monic Laurent polynomial over $\F_q$ of total degree $d$. If $f$ is not the square of a Laurent polynomial then for every $e \in \F_q$ we have
		\begin{equation}\label{e:charsum}
			\left\vert \displaystyle\sum_{x \in \F_q} \eta(ef(x)) \right\vert \leq (d - 1)q^{1/2}.
		\end{equation}
	\end{thm}
	
	In the special case where $f$ is a quadratic polynomial with non-zero discriminant, the following result~\cite{MR1429394} gives an explicit value for the sum in \eref{e:charsum}.
	
	\begin{thm}\label{t:quadweil}
		Let $f \in \F_q[x]$ be a monic, quadratic polynomial with non-zero discriminant. For every $e \in \F_q$ we have
		\[
		\sum_{x \in \F_q} \eta(ef(x)) = -\eta(e).
		\]
	\end{thm}
	
	We can use \tref{t:weil} and \tref{t:quadweil} to prove the following result.
	
	\begin{lem}\label{l:translargeq}
		Suppose that $193 \leq q \equiv 1 \bmod 4$. Then every permutation in the set $\{\alpha[a, b] \in \G : b \in \{2-a, a/(2a-1), -a\}\}$ contains a transposition.
	\end{lem}
	\begin{proof}
		We will prove that every permutation of the form $\alpha[a, 2-a]$ in $\G$ contains a transposition. The remaining claims can be proven using similar arguments.
		
		Let $a \in \F_q$ such that $a(2-a) \in \R_q$. Define 
		\[
		V_a = \{j \in \F_q : \{aj, a^{-1}j-1\} \subseteq \R_q, \{a(j-a+1), a(j-1)\} \subseteq \N_q\}.
		\]
		By \lref{l:2-a}, if $V_a \neq \emptyset$ then $\alpha[a, 2-a]$ contains a transposition. Define 
		\[
		Q(x) = (1+\eta(ax))(1+\eta(a^{-1}x-1))(1-\eta(a(x-a+1)))(1-\eta(a(x-1))).
		\]
		If $x \in V_a$ then $Q(x) = 16$. If $x \in \{0, 1, a, a-1\}$ then $Q(x) \leq 8$. If $x \in \F_q \setminus (V_a \cup \{0, 1, a, a-1\})$ then $Q(x) = 0$. Let $S = \sum_{x \in \F_q} Q(x)$. Then $S \leq 16|V_a| + 32$. Expanding $Q(x)$ and using the fact that $\eta$ is a homomorphism on $\F_q^*$ we can write $S$ as a sum of terms of the form $\sum_{x \in \F_q} \eta(\pm K(x))$ where $K$ is the product of $k$ distinct factors in $\{ax, a^{-1}x-1, a(x-a+1), a(x-1)\}$ for some $k \in \{0, 1, 2, 3, 4\}$. Note that the roots of these factors are distinct because $a \neq 2$. For each $k \in \{1, 2, 3, 4\}$ there are $\binom{4}{k}$ terms $K$ of degree $k$, and \tref{t:weil} or \tref{t:quadweil} applies to each such term. Using these theorems we obtain the bound	
		$S \geq q-11q^{1/2}-6$. As $S \leq 16|V_a| + 32$ it follows that $|V_a| \geq (q-11q^{1/2}-38)/16$, which is positive if $q \geq 193$.
	\end{proof}
	
	We will use \tref{t:weil} and \tref{t:quadweil} in this way many times throughout the paper. 
	To finish the classification of permutations in $\G$ which contain a transposition we used a computer search.
	
	\begin{lem}\label{l:3mod4notrans}
		Suppose that $q \equiv 3 \bmod 4$. The permutation $\alpha[a, b] \in \G$ contains a transposition if and only if
		\[
		(2ab-a-b)(a+b)(a-1) \in \R_q \text{ and } \{2(a+b-2)(a-1), 2a(a+b)\} \subseteq \mathcal{N}_q.
		\]
	\end{lem}
	
	\begin{lem}\label{l:1mod4notrans}
		Suppose that $q \equiv 1 \bmod 4$. The permutation $\alpha[a, b] \in \G$ contains a transposition if and only if one of the following holds:
		\begin{enumerate}[(i)]
			\item $(2ab-a-b)(a+b)(a-1) \in \R_q \text{ and } \{2(a+b-2)(a-1), 2a(a+b)\} \subseteq \mathcal{N}_q$,
			\item $b=2-a$ and $(q, a) \not\in \{(13, 3), (13, 8), (17, 12), (17, 15), (37, 11), (37, 27), (41, 13), (41, 25)\}$,
			\item $b=a/(2a-1)$ and $(q, a) \not\in \{(13, 2), (13, 9), (17, 5), (17, 8), (37, 11), (37, 27), (41, 23), (41, 26)\}$,
			\item $b=-a$ and $(q, a) \not\in \{(13, 7), (13, 11), (17, 3), (17, 11), (37, 10), (37, 26), (41, 12), (41, 17)\}$.
		\end{enumerate}
	\end{lem}
	
	We are now ready to prove \tref{t:n2}.
	
	\begin{proof}[Proof of \tref{t:n2}]
		If $q \equiv 3 \bmod 4$ then the result follows by combining \lref{l:quadrowperms} and \lref{l:3mod4notrans}. Now assume that $q \equiv 1 \bmod 4$. \lref{l:quadrowperms} implies that $\mathcal{L}[a, b]$ contains an intercalate if and only if either $\alpha[a, b]$ contains a transposition or $\alpha[b, a]$ contains a transposition. Since $\{ab, (a-1)(b-1)\} \subseteq \R_q$ for a valid pair $(a, b) \in \F_q^2$, it follows that $(a, b)$ satisfies condition $(i)$ in \lref{l:1mod4notrans} if and only if $(b, a)$ does too. Also note that for each permutation $\alpha[a, b]$ which is an exception in condition $(i)$, $(ii)$ or $(iii)$ in \lref{l:1mod4notrans}, the permutation $\alpha[b, a]$ is not an exception. The result should now be clear.
	\end{proof}
	
	We can also find the number of $N_2$ quadratic Latin squares of order $q$.
	
	\begin{lem}\label{l:numn2}
		The number of $N_2$ quadratic Latin squares of order $q$ is $7q^2/32 + O(q^{3/2})$.
	\end{lem}
	\begin{proof}
		Fix $a \in \F_q \setminus \{-1, 0, 1, 2\}$ and define
		\[
		V_a = \{b \in \F_q : \{ab, (a-1)(b-1), (2ab-a-b)(a+b)(a-1)\} \subseteq \mathcal{R}_q, \{2(a+b-2)(a-1), 2a(a+b)\} \subseteq \mathcal{N}_q\}.
		\]
		If $q \equiv 3 \bmod 4$ then $|V_a|$ is the number of quadratic Latin squares of the form $\L[a, b]$ which contain an intercalate. If $q \equiv 1 \bmod 4$ then the number of squares $\L[a, b]$ which contain an intercalate is $|V_a| + m$ where $m \in \{0, 1, 2, 3\}$. Using \tref{t:weil} and \tref{t:quadweil} in an analogous way as in the proof of \lref{l:translargeq} we can show that
		\[
		\frac1{32}(q-79-58q^{1/2}) \leq |V_a| \leq \frac1{32}(q+79+58q^{1/2}).
		\] 
		The condition that $a \not\in \{-1, 2\}$ is required in order to apply \tref{t:weil} and \tref{t:quadweil} to estimate $|V_a|$. As there are $q-4$ choices for $a \in \F_q \setminus \{-1, 0, 1, 2\}$ it follows that the number of quadratic Latin squares $\mathcal{L}[a, b]$ of order $q$ with $a \not\in \{-1, 2\}$ which contain an intercalate is $q^2/32 + O(q^{3/2})$. Recall that the total number of quadratic Latin squares of order $q$ is $q^2/4 + O(q)$, and there are $O(q)$ quadratic Latin squares of the form $\mathcal{L}[-1, b]$ and $\mathcal{L}[2, b]$. It follows that the number of $N_2$ quadratic Latin squares of order $q$ is $7q^2/32 + O(q^{3/2})$.
	\end{proof}
	
	To conclude this section we describe how to bound the number of permutations in $\G$ which contain a cycle of length $c$, for some $c \in \{2, 3, \ldots, q\}$. Firstly, it is easy to bound the number of permutations in $\G \setminus \G_c$. The comments at the end of \sref{sss:t2} and \sref{sss:a0aa} describe how to bound the number of permutations in $\G_c$ which contain a cycle of length $c$ which is not of Type One. As mentioned at the end of \sref{sss:t1}, \lref{l:cycsseqs} and \lref{l:t1necsuf} give us necessary and sufficient conditions for a permutation in $\G_c$ to contain a Type One cycle of length $c$. To bound the number of permutations which satisfy these conditions we could use \tref{t:weil} and \tref{t:quadweil} in a similar way as used to find the number of $N_2$ quadratic Latin squares. However to apply these theorems in this way we would need to know when products of functions in the set $\{F_{i, z} : i \in \{0, 1, 2, \ldots, 2c-1\}\}$ and products of functions in $\{G_{i, z} : i \in \{0, 1, 2, \ldots, 2c-1\}\}$ are, up to multiplication by a constant, the square of a Laurent polynomial. It seems a difficult task to predict when this occurs.
	
	\section{Row cycles of length $p$ in quadratic Latin squares}\label{ss:p-cycles}
	
	In this section we will prove \tref{t:quadneg}. Throughout this section let $p$ be an odd prime, $d$ a positive integer and $q = p^d$. 
	
	\begin{lem}\label{l:pcycprelim}
		Let $\alpha = \alpha[a, b] \in \G$ with $a \in \R_q$. If there exists $y \in \F_q$ such that $\{y-ja+j, y-(j+1)a+j : j \in \{0, 1, \ldots, p-1\}\} \subseteq \R_q$ then $\alpha$ contains a $p$-cycle.
	\end{lem}
	\begin{proof}
		We will prove by induction on $k \in \{0, 1, 2, \ldots, p\}$ that $\alpha^k(y) = y-ka+k$. The claim is trivial when $k=0$. Suppose that $\alpha^k(y) = y-ka+k$ for some $k \in \{0, 1, 2, \ldots, p-1\}$. Then $\alpha^k(y) \in \R_q$, hence $\varphi^{-1}(\alpha^k(y))-1 = a^{-1}y-(k+1)+ka^{-1} = a^{-1}(y-(k+1)a+k) \in \R_q$ by assumption. Therefore $\alpha^{k+1}(y) = a(a^{-1}(y-(k+1)a+k))+1 = y-(k+1)a+(k+1)$. The lemma follows.
	\end{proof}
	
	In fact, if the hypotheses of \lref{l:pcycprelim} hold then $\alpha$ satisfies the sequence $z \in \{-1, 1\}^{2p}$ defined by $z_i = 1$ for $i \in \{0, 1, 2, \ldots, 2p-1\}$. Analogous arguments allow us to prove the following.
	
	\begin{lem}\label{l:pcycprelim2}
		Let $\alpha = \alpha[a, b] \in \G$ with $a \in \N_q$. If there exists $y \in \F_q$ such that $\{y-ja+j, y-(j+1)a+j : j \in \{0, 1, \ldots, p-1\}\} \subseteq \N_q$ then $\alpha$ contains a $p$-cycle.
	\end{lem}
	
	If the hypotheses of \lref{l:pcycprelim2} hold then $\alpha$ satisfies the sequence $z \in \{-1, 1\}^{2p}$ defined by $z_i = (-1)^{i+1}$ for $i \in \{0, 1, 2, \ldots, 2p-1\}$. We are now ready to prove \tref{t:quadneg}.
	
	\begin{proof}[Proof of \tref{t:quadneg}]
		Let $(a, b) \in \F_q^2$ be valid and let $\alpha = \alpha[a, b]$. First suppose that $\{a, b\} \subseteq \F_p \cap \N_q$. Using \lref{l:junalphr} it is simple to verify that $\alpha_0$ is contained in $\F_p$. Thus $\alpha$ has a cycle of length at most $p$.
		
		Now we deal with the case where $\{a, b\} \not\subseteq \F_p \cap \N_q$. So either $\{a, b\} \subseteq \F_p \cap \R_q$ or $\{a, b\} \not\subseteq \F_p$. We will deal with the latter case first. Since $\mathcal{L}[a, b]$ and $\mathcal{L}[b, a]$ are isomorphic~\cite{MR2134185} and isotopy preserves the lengths of row cycles, we can swap $a$ and $b$ if necessary. Thus we can assume that $a \not\in \F_p$. We will first assume that $a \in \R_q \setminus \F_p$. Define
		\[
		Y = \{y \in \F_q : \{y-ja+j, y-(j+1)a+j\} \subseteq \R_q \text{ for all } j \in \{0, 1, 2, \ldots, p-1\}\}.
		\]
		By \lref{l:pcycprelim}, if $Y \neq \emptyset$ then $\alpha$ contains a $p$-cycle. Define
		\[
		Q(y) = \prod_{j=0}^{p-1}(1+\eta(y-ja+j))(1+\eta(y-(j+1)a+j))
		\]
		and $S = \sum_{y \in \F_q} Q(y)$. If $y \in Y$ then $Q(y) = 4^p$. If $y \in \{ja-1, (j+1)a-1 : j \in \{0, 1, 2, \ldots, p-1\}\}$ then $Q(y) \leq 2^{2p-1}$. In all other cases $Q(y)=0$. It follows that $S \leq 4^p|Y|+p4^p$. Expanding $Q(x)$ and using the fact that $\eta$ is a homomorphism on $\F_q^*$ we can write $S$ as the sum over terms of the form $\sum_{y \in \F_q} \eta(\pm K(y))$ where $K$ is the product of $k$ distinct factors in $\{y-ja+j, y-(j+1)a+j : j \in \{0, 1, 2, \ldots, p-1\}\}$ for some $k \in \{0, 1, 2, \ldots, 2p\}$. If $ja-j = ka-k$ for some $\{j, k\} \subseteq \{0, 1, 2, \ldots, p-1\}$ then $j=k$. Similarly if $(j+1)a-j = (k+1)a-k$ then $j=k$. Suppose that $ja-j = (k+1)a-k$. First note that $j=k+1$ is not a solution to this equation. Hence if this equation is satisfied we must have $a=(j-k)/(j-k-1) \in \F_p$, which is a contradiction. It follows that the roots of each term $K$ are distinct. For each $k \in \{1, 2, \ldots, 2p\}$ there are $\binom{2p}{k}$ terms $K$ of degree $k$, and \tref{t:weil} or \tref{t:quadweil} applies to each such term. Using these theorems we obtain the bound 
		\[
		S \geq q-\binom{2p}{2} - \sum_{k=3}^{2p} \binom{2p}{k} (k-1)q^{1/2} = q + q^{1/2}(p-1)(1-4^p+2p) + p(1-2p).
		\]
		Combining this with the fact that $S \leq 4^p|Y|+p4^p$ gives 
		\[
		|Y| \geq 2^{-2p}(q + q^{1/2}(p-1)(1-4^p+2p) + p(1-2p-4^p)).
		\]
		Therefore $Y \neq \emptyset$ if $q + q^{1/2}(p-1)(1-4^p+2p) + p(1-2p-4^p) > 0$. 
		This inequality will be true if $q^{1/2} > ((1-p)(1-4^p+2p)+((p-1)^2(1-4^p+2p)^2-4p(1-2p-4^p))^{1/2})/2$. Set $q=p^d$ for some positive integer $d$. Then the previous inequality will hold provided that $d > 2\log(((1-p)(1-4^p+2p)+((p-1)^2(1-4^p+2p)^2-4p(1-2p-4^p))^{1/2})/2)/\log(p)$. Define
		\[
		f(p) = \left\lfloor \frac2{\log(p)}\log\left(\frac12((1-p)(1-4^p+2p)+((p-1)^2(1-4^p+2p)^2-4p(1-2p-4^p))^{1/2})\right) \right\rfloor+1.
		\]
		Then we have shown that $\mathcal{L}[a, b]$ contains a row cycle of length $p$ if $d \geq f(p)$. We can use analogous arguments in conjunction with \lref{l:pcycprelim2} to prove that same result if $a \in \N_q \setminus \F_p$.
		
		We now deal with the case where $\{a, b\} \subseteq \F_p \cap \R_q$. By \lref{l:pcycprelim}, to show that $\alpha$ contains a $p$-cycle it suffices to show that there exists some $y \in \F_q$ such that $\{y-ja+j, y-(j+1)a+j : j \in \{0, 1, \ldots, p-1\}\} \subseteq \R_q$. Note that $y-ja+j = y-(k+1)a+k$ where $k=j+a/(1-a)$. Therefore the result will follow if
		$\{y-ja+j : j \in \{0, 1, \ldots, p-1\}\} \subseteq \R_q$. We can use analogous arguments as in the case where $a \in \R_q \setminus \F_p$ to show that $\alpha$ contains a $p$-cycle if $q=p^d$ with $d > 2\log(((2-p)(1-2^p+p)+((p-2)^2(1-2^p+p)^2-8p(1-p-2^p))^{1/2})/2)/\log(p)$. As this quantity is less than $f(p)$ it follows that $\mathcal{L}[a, b]$ contains a row cycle of length $p$ if $d \geq f(p)$.
	\end{proof}
	
	The function $f$ provided in the proof of \tref{t:quadneg} satisfies $f(p) \sim p\log(16)/\log(p)$. Furthermore, $f$ is not minimal. For example, $f(3) = 9$, however every quadratic Latin square of order $3^d$ contains a $3$-cycle if $d \geq 7$.
	
	\section{Anti-perfect $1$-factorisations and anti-atomic Latin squares}\label{ss:antiperf}
	
	In this section we prove our main results concerning anti-perfect $1$-factorisations and anti-atomic Latin squares. To prove \tref{t:antiperf} and \tref{t:antiatom} we need the following definition. Let $v$ be a positive integer and $K \subseteq \{2, 3, 4, \ldots\}$. A \emph{pairwise balanced design} $\pbd(v, K)$ is a pair $(X, \B)$ where $X$ is a set of order $v$ whose elements are called points, and $\B$ is a collection of subsets of $X$ called blocks, such that the size of each block in $\B$ is an element of $K$, and each pair of distinct points in $X$ appears in exactly one block in $\B$.
	
	We can use PBD's to construct Latin squares, in a method known as the `PBD construction', which we describe now. It is known that an idempotent Latin square of order $n$ exists for all $n \neq 2$. Suppose that $(X, \B)$ is a $\pbd(v, K)$ for some positive integer $v$ and some set $K \subseteq \{3, 4, 5, \ldots, \}$. For each $B \in \B$ let $L^B$ be an idempotent Latin square with symbol set $B$. We can then define a $v \times v$ idempotent Latin square $L$ with symbol set $X$ by
	\[
	L_{i, j} = \begin{cases}
		i & \text{if } i = j, \\
		L^B_{i, j} & \text{if } i \neq j, \text{ where } B \text{ is the unique block in } \B \text{ with } \{i, j\} \subset B.
	\end{cases}
	\]
	The PBD construction has been used to solve various problems, such as the construction of mutually orthogonal Latin squares (see e.g.~\cite{MR2469212}) and the construction of $1$-factorisations which contain only short cycles~\cite{MR3537912, MR2433008, MR2475030}. We now give a series of simple lemmas regarding Latin squares obtained from the PBD construction. The first is a simple observation.
	
	\begin{lem}\label{l:pbdconj}
		Any conjugate of a Latin square obtained from the PBD construction can also be obtained from the PBD construction.
	\end{lem}
	
	The following is a known result~\cite{MR2433008}.
	
	\begin{lem}\label{l:pbdcyc}
		Let $L$ be a Latin square obtained from the pairwise balanced design $(X, \B)$.
		Let $\{i, j\} \subseteq X$ and let $B \in \B$ be the block containing $i$ and $j$. Then $r_{i, j}(B) = B$.
	\end{lem}
	
	\begin{lem}\label{l:pbdideminv}
		Let $L$ be a Latin square obtained from the pairwise balanced design $(X, \B)$. If $L^B$ is involutory for each $B \in \B$ then $L$ is also involutory.
	\end{lem}
	\begin{proof}
		Let $\{i, j\} \subseteq X$ with $i \neq j$. Then $k = L_{i, j} = L^B_{i, j}$ where $B \in \B$ is the block containing $i$ and $j$. Hence $k \in B$ also. Furthermore $L^B$ is idempotent, thus $k \neq i$. Thus $L_{i, k} = L^B_{i, k} = j$ because $L^B$ is involutory. Therefore $L$ is involutory as well.
	\end{proof}
	
	It is known that an idempotent, involutory Latin square of order $n$ exists if and only if $n$ is odd. Combining \lref{l:pbdconj}, \lref{l:pbdcyc} and \lref{l:pbdideminv} we obtain the following corollary.
	
	\begin{cor}\label{c:pbdanti}
		Let $K \subseteq \{3, 4, 5, \ldots, \}$ and suppose that there exists a $\pbd(v, K)$ with at least two blocks. Then there exists an anti-atomic Latin square of order $v$. Furthermore if $K$ contains only odd integers, then there exists an anti-perfect $1$-factorisation of $K_{v+1}$.	
	\end{cor}
	
	A result of Colbourn, Haddad and Linek~\cite{MR1370131} implies the existence of a $\pbd(v, K)$ with at least two blocks where $K$ contains only odd integers whenever $v \geq 7$ is odd. Combining this with \cyref{c:pbdanti} proves \tref{t:antiperf}. As mentioned in \sref{s:intro}, Dukes and Ling~\cite{MR2475030} constructed a $1$-factorisation of $K_{v+1}$ whose cycles are all of length at most $1720$, for all odd $v$. Each of these $1$-factorisations comes from the PBD construction with a $\pbd(v, \{3, 5\})$. \lref{l:pbdcyc} tells us that these $1$-factorisations are actually anti-perfect for all $v \geq 7$. We are now ready to prove \tref{t:antiatom}.
	
	\begin{proof}[Proof of \tref{t:antiatom}]
		A result of Hartman and Heinrich~\cite{MR1209190} implies the existence of a $\pbd(v, K)$ with at least two blocks where $2 \not\in K$ whenever $v = 7$ or $v \geq 9$. \cyref{c:pbdanti} then implies that there exists an anti-atomic Latin square of order $v$ whenever $v \not\in \{1, 2, 3, 4, 5, 6, 8\}$. Let $L$ be a Latin square which is derived from the Cayley table of a group $G$. By~\cite[Theorem 4.2.2]{MR0351850}, every conjugate of $L$ is isotopic to itself. Every row cycle in the row permutation $r_{g, h}$ of $L$ has length equal to the order of $gh^{-1}$ in $G$. Thus the existence of a non-cyclic group of order $v$ implies the existence of an anti-atomic Latin square of order $v$. This proves that anti-atomic Latin squares of order $v$ exist for $v \in \{4, 6, 8\}$. It is easy to verify that every Latin square of order $v \in \{2, 3, 5\}$ contains a row cycle of length $v$, and thus is not anti-atomic.
	\end{proof}
	
	\lref{l:pbdconj} and \lref{l:pbdcyc} imply that anti-atomic Latin squares can be built from $1$-factorisations constructed in~\cite{MR3537912, MR2433008, MR2475030}. We also record that Latin squares corresponding to Steiner $1$-factorisations give us examples of anti-atomic Latin squares of order $q$ for all $q \equiv 1 \bmod 6$ or $3 < q \equiv 3 \bmod 6$.
	
	The existence spectrum of anti-atomic Latin squares is the same as the existence spectrum of anti-perfect $1$-factorisations of complete bipartite graphs. If $L$ is an anti-atomic Latin square of order $n$ then $\mathcal{E}(L)$ is an anti-perfect $1$-factorisation of $K_{n, n}$. However the converse is not true in general. For example, let $L$ be the quadratic Latin square $\mathcal{L}[2, 6]$ of order $11$. Then $\mathcal{E}(L)$ is anti-perfect but the $(2, 1, 3)$-conjugate of $L$ is row-Hamiltonian.
	
	The remainder of this section will be devoted to proving \tref{t:quadanti}. The first step is to find for which prime powers $q$ there exists quadratic, idempotent, involutory Latin squares of order $q$ which do not contain any row cycle of length $q$.
	Every quadratic Latin square is idempotent. The following lemma~\cite{MR623318} gives sufficient conditions for $\mathcal{L}[a, b]$ to be involutory.
	
	\begin{lem}\label{l:qii}
		Let $q \equiv 3 \bmod 4$ be a prime power and let $a \in \N_q \setminus \{-1\}$. The Latin square $\mathcal{L}[a, a^{-1}]$ is involutory.
	\end{lem}
	
	Every conjugate of a quadratic Latin square is also a quadratic Latin square, and a result of Wanless~\cite{MR2134185} allows us to determine these conjugates. Using this result, it is simple to verify that the only involutory quadratic Latin squares which are not given by \lref{l:qii} are those squares of the form $\mathcal{L}[a, a]$. As mentioned in \sref{ss:n2}, Latin squares $\mathcal{L}[a, a]$ of order $q$ are isotopic to the Cayley table of the additive group $(\F_q, +)$. 
	
	The following is a corollary of \tref{t:n2}. 
	
	\begin{cor}\label{c:iii}
		Let $q \equiv 3 \bmod 4$ be a prime power and let $a \in \N_q$. The Latin square $\mathcal{L}[a, a^{-1}]$ contains an intercalate if and only if $\{2(1-a), 2(a^2+1)\} \subseteq \N_q$.
	\end{cor}
	\begin{proof}
		\tref{t:n2} implies that $\mathcal{L}[a, a^{-1}]$ contains an intercalate if and only if $(2-a-a^{-1})(a+a^{-1})(a-1) \in \R_q$ and $\{2(a+a^{-1}-2)(a-1), 2a(a+a^{-1})\} \subseteq \N_q$. This is equivalent to the condition that $\{2(a+a^{-1}-2)(a-1), 2a(a+a^{-1})\} \subseteq \N_q$ because $2(a+a^{-1}-2)(a-1) \cdot 2a(a+a^{-1}) = 4a(a+a^{-1}-2)(a+a^{-1})(a-1)$. This is equivalent to $\{2(1-a), 2(a^2+1)\} \subseteq \N_q$ because $a+a^{-1}-2 = a^{-1}(a-1)^2$.
	\end{proof}
	
	Let $q \equiv 3 \bmod 4$ be a prime power. Dinitz and Dukes~\cite{MR2206402} studied $1$-factorisations of complete graphs $K_{q+1}$ of the form $\mathcal{F}(\mathcal{L}[a, a^{-1}])$ for $a \in \N_q \setminus \{-1\}$. Among other things, they characterised when such $1$-factorisations contain a cycle of length four. Their Theorem $3.2$ is equivalent to \cyref{c:iii}. V\'{a}zquez-\'{A}vila~\cite{MR4354936} also studied $1$-factorisations of the form $\mathcal{F}(\mathcal{L}[a, a^{-1}])$ and showed that if $q \equiv 11 \bmod 24$ then there is a $1$-factorisation $\mathcal{F}$ of this form such that every pair of $1$-factors in $\mathcal{F}$ induces a subgraph in $K_{q+1}$ which contains exactly one cycle of length four.
	
	Let $L$ be a quadratic, idempotent, involutory Latin square of prime power order $q \equiv 3 \bmod 4$ which contains an intercalate. \lref{l:quadrowperms} implies that every row permutation of $L$ contains a transposition.
	A standard application of \tref{t:weil} and \tref{t:quadweil}, in conjunction with \cyref{c:iii}, shows that such a square exists for all $q \geq 83$. 
	A computer search then allows us to prove the following result.
	
	\begin{lem}\label{l:notrh}
		Let $q \equiv 3 \bmod 4$ be a prime power with $q > 3$. There exists a quadratic, idempotent, involutory Latin square $L$ of order $q$ such that no row permutation of $L$ is a $q$-cycle.
	\end{lem}
	
	\lref{l:notrh} proves the existence of an anti-perfect $1$-factorisation of $K_{q+1}$ for any prime power $q$ where $3 < q \equiv 3 \bmod 4$. The next step in proving \tref{t:quadanti} is to prove some simple results concerning the direct product of Latin squares.
	
	\begin{lem}\label{l:dpcyc}
		Let $L$ and $M$ be Latin squares. Let the set of lengths of row cycles in $L$ and the set of lengths of row cycles in $M$ be $R$ and $P$, respectively. The set of lengths of row cycles in $L \times M$ is
		\[
		R \cup P \cup \{\text{\rm lcm}(r, p) : r \in R, p \in P\}.
		\]
	\end{lem}
	\begin{proof}
		Let the symbol sets of $L$ and $M$ be $S$ and $T$, respectively. For $i, j \in S$ denote the row permutation of $L$ mapping row $i$ to row $j$ by $u_{i, j}$. Similarly for $k, \ell \in T$ denote the row permutation of $M$ mapping row $k$ to row $\ell$ by $v_{k, \ell}$. For the purposes of this proof we will say that the permutations $u_{i, i}$ and $v_{k, k}$ denote the identity permutation, for $i \in S$ and $k \in T$. Let $\{(i, k), (j, \ell)\} \subseteq S \times M$ and consider the row permutation $r_{(i, k), (j, \ell)}$ of $L \times M$. We will show that $r_{(i, k), (j, \ell)} = u_{i, j} \times v_{k, \ell}$. Let $x \in S \times T$. Then $x = (L \times M)_{(i, k), (a, b)}$ for some $a \in S$ and $b \in T$. So
		\[
		\begin{aligned}
			r_{(i, k), (j, \ell)}(x) &= r_{(i, k), (j, \ell)}((L \times M)_{(i, k), (a, b)}) \\
			&= (L \times M)_{(j, \ell), (a, b)} \\
			&= (L_{j, a}, M_{\ell, b}) \\
			&= (u_{i, j}(L_{i, a}), v_{k, \ell}(M_{k, b})) \\
			&= u_{i, j} \times v_{k, \ell} (L_{i, a}, M_{k, b}) \\
			&= u_{i, j} \times v_{k, \ell}(x).
		\end{aligned}
		\]
		Hence $r_{(i, k), (j, \ell)} = u_{i, j} \times v_{k, \ell}$ as claimed. So the length of the cycle in $r_{(i, k), (j, \ell)}$ containing $x = (y, z)$ is the lowest common multiple of the lengths of the row cycle of $u_{i, j}$ containing $y$ and the row cycle of $v_{k, \ell}$ containing $z$. The lemma should now be clear.
	\end{proof}
	
	\begin{cor}\label{c:antiatomprod}
		Let $L$ and $M$ be Latin squares, such that at least one of $L$ or $M$ is anti-atomic. Then $L \times M$ is anti-atomic.
	\end{cor}
	\begin{proof}
		Let the symbol sets of $L$ and $M$ be $S$ and $T$, respectively. \lref{l:dpcyc} implies that if $L$ does not contain a row cycle of length equal to $|S|$ or $M$ does not contain a row cycle of length equal to $|T|$ then $L \times M$ does not contain a row cycle of length equal to $|S||T|$. It is a simple task to verify that the $(x, y, z)$-conjugate of $L \times M$ is equal to the direct product of the $(x, y, z)$-conjugate of $L$ and the $(x, y, z)$-conjugate of $M$, for any $1$-line permutation $(x, y, z)$ of $\{1, 2, 3\}$. The result follows.
	\end{proof}
	
	\begin{lem}\label{l:dpsym}
		Let $L$ and $M$ be idempotent, involutory Latin squares. Then $L \times M$ is also idempotent and involutory.
	\end{lem}
	\begin{proof}
		Let the symbol sets of $L$ and $M$ be $S$ and $T$, respectively. Let $(x, y) \in S \times T$. We know that $(L \times M)_{(x, y), (x, y)} = (L_{x, x}, M_{y, y}) = (x, y)$ because $L$ and $M$ are idempotent. Hence $L \times M$ is idempotent. Let $\{(u, v), (c, d)\} \subseteq S \times T$ be such that $(L \times M)_{(u, v), (x, y)} = (c, d)$. Then $L_{u, x} = c$ and hence $L_{u, c} = x$ because $L$ is involutory. Similarly $M_{v, d} = y$. Thus $(L \times M)_{(u, v), (c, d)} = (x, y)$ and so $L \times M$ is involutory.
	\end{proof}
	
	We are now ready to prove \tref{t:quadanti}.
	
	\begin{proof}[Proof of \tref{t:quadanti}]
		We first prove that there exists an anti-atomic quadratic Latin square of odd prime power order $q$ if and only if $q \not\in \{3, 5\}$. Let $q$ be an odd prime power. If $q \equiv 3 \bmod 4$ then by \lref{l:quadrowperms}, a quadratic Latin square of order $q$ contains an intercalate if and only if all of its row permutations contain a transposition. As noted in the proof of \tref{t:n2}, a valid pair $(a, b) \in \F_q^2$ satisfies condition $(i)$ in \lref{l:1mod4notrans} if and only if $(b, a)$ also does.
		Combining this with \lref{l:quadrowperms} implies that if $q \equiv 1 \bmod 4$ and $L$ is a quadratic Latin square of order $q$ which does not correspond to one of the exceptions listed in \lref{l:1mod4notrans}, then $L$ contains an intercalate if and only if all of its row permutations contain a transposition.
		The property of a Latin square containing an intercalate is invariant under conjugation. Therefore if a quadratic Latin square contains an intercalate, then either it corresponds to one of the exceptions in \lref{l:1mod4notrans}, or it is anti-atomic. 
		The comment before \lref{l:notrh} tells us that a quadratic Latin square of order $q$ containing an intercalate exists if $83 \leq q \equiv 3 \bmod 4$. We can then use a computer to confirm that an anti-atomic quadratic Latin square of order $q \equiv 3 \bmod 4$ exists whenever $3 < q \leq 79$. If $q \equiv 1 \bmod 4$ then we know from \tref{t:n2} that all Latin squares $\mathcal{L}[a, -a]$ of order $q$ contain an intercalate, and only eight of these correspond to the exceptions in \lref{l:1mod4notrans}. We can use this fact, along with \tref{t:weil} and \tref{t:quadweil} in the standard way, to prove that there exists an anti-atomic quadratic Latin square of order $q$ for all $5 < q \equiv 1 \bmod 4$.
		This proves that there exists an anti-atomic quadratic Latin square of odd prime power order $q$ if and only if $q \not\in \{3, 5\}$. Let $n \not\in \{3, 5, 15\}$ be an odd integer with prime power factorisation $q_1 \cdot q_2 \cdots q_k$. Then $q_i \not\in \{3, 5\}$ for some $i \in \{1, 2, \ldots, k\}$. The existence of an anti-atomic quadratic Latin square of order $q_i$ combined with \cyref{c:antiatomprod} proves the existence of an anti-atomic Latin square of order $n$ which is the direct product of quadratic Latin squares.
		
		We now prove the second claim of \tref{t:quadanti}. Let $n$ be an odd integer which contains a prime power divisor $m \neq 3$ with $m \equiv 3 \bmod 4$. Let the prime power factorisation of $n$ be $m \cdot q_2 \cdots q_k$. For each $q_i$ with $i \in \{2, 3, \ldots, k\}$ there exists a quadratic, idempotent, involutory Latin square of order $q_i$. In particular, any square of the form $\mathcal{L}[a, a]$ with $a \in \F_q \setminus \{0, 1\}$ satisfies this property. Also, \lref{l:notrh} implies the existence of a quadratic, idempotent, involutory Latin square of order $m$ which does not contain a row cycle of length $m$. Combining these facts with \lref{l:dpcyc} and \lref{l:dpsym} proves the claim.	
	\end{proof}
	
	\section{Conclusion}\label{s:conc}
	
	In \sref{ss:n2} we characterised exactly when a quadratic Latin square is $N_2$. Note that $N_2$ quadratic Latin squares of order $q$ exist for all odd prime powers $q$, because the squares $\mathcal{L}[a, a]$ are $N_2$ for any $a \in \F_q \setminus \{0, 1\}$. We also found that there are $7q^2/32 + O(q^{3/2})$ quadratic $N_2$ Latin squares of order $q$. Dr\'{a}pal and Wanless~\cite{isoquas} showed that the quadratic Latin squares $\mathcal{L}[a, b]$ and $\mathcal{L}[c, d]$ of order $q$ are isomorphic if and only if $\{c, d\} = \{\theta(a), \theta(b)\}$ for an automorphism $\theta$ of $\mathbb{F}_q$. It follows that the number of isomorphism classes of $N_2$ quadratic Latin squares of order $q$ is at least $\Theta(q^2/\log(q))$. \lref{l:dpcyc} implies that the direct product of $N_2$ Latin squares is also $N_2$. This fact was known by McLeish~\cite{MR396298}. This means that we can construct $N_2$ Latin squares of order $n$ for all odd $n$ by using the direct product of quadratic Latin squares. In fact, we can construct many $N_2$ Latin squares for all odd orders. However, this construction only gives a small number of $N_2$ Latin squares, in comparison to the total number of $N_2$ Latin squares. Kwan, Sah, Sawhney and Simkin~\cite{latinsub} have used a probabilistic argument to show that there are at least $(e^{-9/4}n-o(n))^{n^2}$ Latin squares of order $n$ which are devoid of intercalates.
	
	\tref{t:quadneg} tells us that quadratic Latin squares of order $q = p^d$ will not be useful for constructing perfect $1$-factorisations unless $d$ is small. It also tells us that unless $d$ is small, the only quadratic Latin squares which could be useful for constructing $1$-factorisations which contain only short cycles are the squares $\mathcal{L}[a, b]$ with $\{a, b\} \subseteq \F_p \cap \N_q$. Combining \lref{l:quadrowperms} with the fact that every row permutation of such a square contains a cycle of length at most $p$ makes it tempting to investigate these squares when searching for $1$-factorisations which contain only short cycles. However computational evidence seems to suggest that such Latin squares always contain some large row cycles.
	
	\section*{Acknowledgements}
	
	I would like to thank Ian Wanless for many useful discussions and for helpful comments on an early draft of this paper. The author was supported by an Australian Government Research Training Program Scholarship.
	
	\printbibliography
	
\end{document}